\documentclass[aps,groupedaddress, letterpaper, twocolumn,secnumarabic]{revtex4}
\usepackage[letterpaper,margin=1in]{geometry}                
\usepackage{amssymb,amsmath,amsfonts,mathscinet,stmaryrd}
\usepackage{xcolor}
\usepackage[colorlinks,
             linkcolor=blue,
             citecolor=green!70!black,
             pdfproducer={pdfLaTeX},
             pdfpagemode=None,
             bookmarksopen=true
             bookmarksnumbered=true]{hyperref}

\usepackage{tikz}
\usetikzlibrary{decorations.markings,arrows,decorations.pathreplacing}
\usepackage{arydshln}

\allowdisplaybreaks

\newcommand\arXiv[1]{\href{http://arxiv.org/abs/#1}{\nolinkurl{arXiv:#1}}}
\newcommand\MRnumber[1]{\href{http://www.ams.org/mathscinet-getitem?mr=#1}{\nolinkurl{MR#1}}}
\newcommand\DOI[1]{\href{http://dx.doi.org/#1}{\nolinkurl{DOI:#1}}}
\newcommand\MAILTO[1]{\href{mailto:#1}{\nolinkurl{#1}}}

\usepackage{amsthm}

\newtheorem{theorem}{Theorem}

\newtheorem{dummy}{Dummy}[section]
\newtheorem{corollary}[dummy]{Corollary}
\newtheorem{definition}[dummy]{Definition}
\newtheorem{lemma}[dummy]{Lemma}
\newtheorem{proposition}[dummy]{Proposition}

\theoremstyle{definition}
\newtheorem{remarknodiamond}[dummy]{Remark}
\newenvironment{remark}{\begin{remarknodiamond}}{\hfill\ensuremath\Diamond\end{remarknodiamond}}

\usepackage{dsfont}
\renewcommand\mathbb\mathds

\newcommand\bC{\mathbb C}

\newcommand\bG{\mathbb G}
\newcommand\bH{\mathbb H}

\newcommand\bN{\mathbb N}

\newcommand\bR{\mathbb R}

\newcommand\bZ{\mathbb Z}

\newcommand\cA{\mathcal A}
\newcommand\cB{\mathcal B}
\newcommand\cC{\mathcal C}

\newcommand\cF{\mathcal F}

\newcommand\cM{\mathcal M}

\newcommand\cQ{\mathcal Q}

\newcommand\cS{\mathcal S}

\newcommand\cW{\mathcal W}
\newcommand\cX{\mathcal X}
\newcommand\cY{\mathcal Y}
\newcommand\cZ{\mathcal Z}

\newcommand\rB{\mathrm B}

\usepackage[mathscr]{euscript}

\DeclareMathOperator\homology{H}
\renewcommand\H{\homology}

\renewcommand\d{\mathrm d}

\newcommand\fd{{\mathrm{fd}}}

\newcommand\bk{{\mathbb{k}}}

\newcommand\longto\longrightarrow
\newcommand\mono\hookrightarrow
\newcommand\epi\twoheadrightarrow
\newcommand\isom{\overset\sim\to}
\newcommand\<\langle
\renewcommand\>\rangle

\newcommand\sminus\smallsetminus

\newcommand\condense{\mathrel{\,\hspace{.75ex}\joinrel\rhook\joinrel\hspace{-.75ex}\joinrel\rightarrow}}

\tikzset{
    dot/.style={circle,draw,fill,inner sep=1pt},
}

\newcommand\Mod{\cat{Mod}}
\DeclareMathOperator\SH{SH}

\newcommand\op{\mathrm{op}}
\newcommand\id{\mathrm{id}}

\DeclareMathOperator\Sq{Sq}

\newcommand\Vect{\cat{Vec}}
\newcommand\SVec{\cat{SVec}} \newcommand\SVect\SVec

\DeclareMathOperator\Spec{Spec}
\DeclareMathOperator\Kar{Kar}
\DeclareMathOperator\End{End}
\DeclareMathOperator\Gal{Gal}

\DeclareMathOperator\Cliff{Cliff}

\DeclareMathOperator\Mor{Mor}

\DeclareMathOperator\ob{ob}

\DeclareMathOperator\Mat{Mat}

\newcommand\PGL{\mathrm{PGL}}

\newcommand\pt{\mathrm{pt}}

\newcommand\define[1]{\emph{#1}}
\newcommand\cat[1]{\mathbf{#1}}

\begin{document}
\title{On the classification of topological orders}
\author{Theo Johnson-Freyd}
\thanks{I thank D.~Freed, D.~Gaiotto, L.~Kong, D.~Reutter, and M.~Yu, and in particular 
the anonymous referee, for discussion and comments.
I thank C.~Scheimbauer for  Remark~\ref{remark.bestiary}, and D.~Nykshych for the end of Remark~\ref{remark.fermioniocWittgroup}.
Research at Perimeter Institute is supported by the Government of Canada through the Department of Innovation, Science and Economic Development Canada and by the Province of Ontario through the Ministry of Colleges and Universities.}
\affiliation{Perimeter Institute for Theoretical Physics}
\affiliation{Department of Mathematics, Dalhousie University}

\begin{abstract}
\textbf{Abstract.} 
We axiomatize the extended operators in topological orders (possibly gravitationally anomalous, possibly with degenerate ground states) in terms of monoidal Karoubi-complete $n$-categories which are mildly dualizable and have trivial centre. Dualizability encodes the word ``topological,'' and we take it as the definition of ``(separable) multifusion $n$-category''; triviality of the centre implements the physical principle of ``remote detectability.'' We show that such $n$-categorical algebras are Morita-invertible (in the appropriate higher Morita category), thereby identifying topological orders with anomalous fully-extended TQFTs. We identify centreless fusion $n$-categories (i.e.\ multifusion $n$-categories with indecomposable unit) with centreless braided fusion $(n{-}1)$-categories.
We then discuss the classification in low spacetime dimension, proving in particular that all $(1{+}1)$- and $(3{+}1)$-dimensional topological orders, with arbitrary symmetry enhancement, are suitably-generalized topological sigma models. These mathematical results confirm and extend a series of conjectures and results by L.~Kong, X.G.~Wen, and their collaborators.
\end{abstract}

\maketitle

\tableofcontents

\section{Introduction and summary of results}

The classification of gapped topological phases of matter remains of fundamental interest \cite{Nobel}.
The goal of this article is to give an \emph{a priori} analysis of the algebras of operators in such phases. 
Necessary for such an analysis is to propose axioms/definitions for the physical objects of study. 
Our definition builds on the proposals of L.\ Kong, X.G.~Wen, and their collaborators \cite{MR2942952,1405.5858,10.1093/nsr/nwv077,KWZ1,KWZ2,PhysRevX.8.021074,PhysRevX.9.021005,PhysRevB.100.045105,1912.13492,KLWZZ}:
\begin{definition}\label{defn.main}
  An \define{$(n{+}1)$-dimensional topological order} is a multifusion $n$-category with trivial centre.
\end{definition}
Our topological orders correspond to the ``closed'' topological orders of \cite{1405.5858,KWZ1,KWZ2}. Those papers sometimes identify the words ``closed'' and ``nonanomalous,'' but we will see (Theorem~\ref{thm.anomalousTQFT}) that the topological orders of Definition~\ref{defn.main} do sometimes have (gravitational 't~Hooft) anomalies.
(Those papers contemplate a ``cochain complex'' whose ``$n$-cochains'' are multifusion $n$-categories, and whose differential takes a cochain to its centre. Anomalies in the sense of \cite{MR3165462} correspond to the failure of a topological order to be ``exact.'')  In $(3{+}1)$-dimensions there are known gapped quantum systems which are not topological \cite{PhysRevLett.94.040402,PhysRevA.83.042330}; such phases are not part of our analysis, and will not be mentioned further, but one can hope that some of the ideas in this article might apply to those phases as well.

We do not propose a mathematical definition of the term ``(gapped topological) phase of matter,'' but rather adopt the name ``topological order'' from \cite{MR1043302}, where it is introduced without axiomatization. Indeed, topological orders as axiomatized by Definition~\ref{defn.main} are \emph{not} equivalent to (gapped topological) phases of matter, but rather correspond to equivalence classes of such phases under the relation of stacking with invertible phases.
Physically, the multifusion $n$-category contains and organizes all of the operators in the topological order: its objects are the operators of codimension-$1$, its $1$-morphisms are the interfaces (of codimension-$2$ in spacetime) between codimension-$1$ operators, and so on.
The higher categorical nature of (extended) operators is well established \cite{MR2827874}.
  Definition~\ref{defn.main} asserts that a topological order is fully determined by its ``algebra'' of operators (and that this ``algebra'' satisfies certain axioms). Depending on one's ontology, this Definition either defines an interesting equivalence class of mathematical/physical systems to study, or it is a physical assertion which may or may not be true for ``physical topological orders.''

The articles cited above do not provide a mathematically precise definition of (multi)fusion $n$-category. 
(Multi)fusion $2$-categories are defined in \cite{Reutter2018}.
Providing  a definition when $n>2$ is the first contribution of this paper, and is the main focus of \S\ref{sec.dualizability}; actually, we only propose a definition of \emph{separable} (multi)fusion $n$-category, but in characteristic $0$ all fusion $1$-categories are separable \cite{DSPS}, so  perhaps separability is part of what a physicist means by ``fusion category.'' 
(Or maybe not: the definition in \cite{Reutter2018} is stated only over $\bC$, but seems to compile over any field and does not seem to impose separability.)
 In order to formulate the definition, we rely heavily on constructions and results from \cite{GJFcond}, which, under the name ``categorical condensation,'' generalizes Karoubi (aka idempotent) completion to higher categories; we review the main definitions in \S\ref{subsec.condensation}. \S\ref{subsec.mor1} and \S\ref{subsec.dualizability} provide the machinery needed for our definition: a \define{multifusion $n$-category} is a ($\bC$-linear and additive) Karoubi-complete monoidal $n$-category satisfying a mild dualizability criterion in the higher Morita category $\cat{Mor}_1(n\cat{KarCat})$ of such objects. Our first main result, Theorem~\ref{thm.dualizability}, says that this mild dualizability criterion in fact implies a condition called ``full dualizability.'' Full dualizability is a strong finiteness condition which has appeared most prominently in the study of fully extended topological quantum field theories \cite{BaeDol95,Lur09}, and we will use mathematical results from that subject, but we do not assume any \emph{a priori} identification between topological orders and TQFTs.

General multifusion $n$-categories do not determine topological orders. Building on \cite{MR2942952,PhysRevX.3.021009},  \cite{1405.5858} emphasizes the importance of the ``principle of remote detectability,'' which roughly speaking asserts that there are no ``invisible'' operators. The way operators ``see'' each other is through a higher-categorical version of the commutator, hence our requirement that the multifusion $n$-category should have trivial centre. The requirement is familiar in low dimensions: for example, in $2{+}1$ dimensions, it asserts that the braided fusion category of line operators is nondegenerate (which together with a ribbon structure defines the notion of \define{modular} tensor category). Our second main result, Theorem~\ref{thm.invertibility} in \S\ref{sec.centre}, says that multifusion $n$-categories with trivial centre are automatically \define{Morita invertible}. 
Theorem~\ref{thm.anomalousTQFT} is an immediate consequence: topological orders are equivalent to \define{anomalous fully-extended TQFTs} in the sense of~\cite{MR3165462},  confirming mathematically the main conjecture of~\cite{1405.5858}.

We return to a physical discussion in \S\ref{sec.justification}. In \S\ref{subsec.correct}
we list carefully the ans\"atze leading to our proposed definition of ``topological order,'' and comment on potential challenges each ansatz faces: maybe ``physical'' topological orders are simply more subtle than straightforward category theory predicts. Moreover, there is an important difference between the topological orders as axiomatized in this paper and other notions of ``topological phases of matter.'' Namely, \define{invertible} topological phases all produce the trivial topological order, and so our topological orders at best describe some quotient $\{\text{topological phases}\}/\{\text{invertible phases}\}$. Indeed, our classification is only possible because of this quotient. The exact classification of invertible phases is a subtle question (see e.g.\ \cite{Wen2013,KitaevTalk2,MR3978827}), and is almost surely coarser than would be predicted from a purely algebraic analysis~\cite{2104.04534}. In \S\ref{subsec.complete} we do prove that our topological orders capture the quotient $$\{\text{anomalous TQFTs}\} = \frac{\{\text{TQFTs}\}}{\{\text{invertible TQFTs}\}},$$ and provide a nonrigorous construction which, if it can be made rigorous, would imply that our topological orders also capture the quotient $$\{\text{topological orders}\} = \frac{\{\text{topological phases}\}}{\{\text{invertible phases}\}}.$$

Section~\ref{sec.braided} explains why codimension-$1$ operators are rarely considered in the study of topological orders, and in particular why \cite{1405.5858} focuses on \define{braided} fusion, rather than fusion, $n$-categories. Let us say that a multifusion $n$-category is \define{fusion} if it has no nontrivial $0$-dimensional operators, in the sense that the only $n$-morphisms from (the identity $(n{-}1)$-morphism on) the identity object to itself are multiples of the identity. Physically, this condition corresponds to the assumption that the local ground state in the topological order is nondegenerate.
Theorem~\ref{thm.braided}, proved in \S\ref{subsec.fusion}, says that such a topological order is determined by its operators of codimension-$2$ and higher. Specifically, that Theorem, together with Corollary~\ref{cor.braidedcentre},
provides an equivalence between fusion $n$-categories with trivial centre and braided fusion $(n{-}1)$-categories with trivial centre 
(establishing Conjecture~2.18 of \cite{KWZ1}). 
After a discussion of higher centres in \S\ref{subsec.braidedcentre}, we explain in \S\ref{subsec.nolines} a generalization of 
 Theorem~\ref{thm.braided}: an $(n{+}1)$-dimensional topological order with no nontrivial operators of dimension $<k$ is fully determined by its $(n-k)$-category of operators of codimension $>k$. This is not to say that there are no operators of smaller codimension, but merely that they are all built, in a canonical way, from the higher-codimension operators. The $k=2$ case, spelled out in detail in Theorem~\ref{thm.nolines}, is the primary unproven ingredient in \cite{PhysRevX.8.021074,PhysRevX.9.021005,PhysRevB.100.045105}.

We end with an analysis in \S\ref{sec.dimbydim} of topological orders in low dimension. 
\S\ref{subsec.0+1d} observes that $(0{+}1)$-dimensional topological orders are what the mathematicians call \define{central simple algebras}, and that a time-reversal structure can be understood as \define{Galois descent data} from $\bC$ to $\bR$.  \S\ref{subsec.1+1d} classifies $(1{+}1)$-dimensional topological orders. Bosonically without any symmetry enhancement, the classification is very simple, but the general case is rich: in Theorem~\ref{thm.1+1d} we classify (bosonic or fermionic) $(1{+}1)$-dimensional topological orders with $G$ flavour symmetry in terms of \define{reduced equivariant (super)cohomology}; if the symmetry is time-reversing, then one must use \define{twisted} (reduced equivariant super)cohomology, or equivalently (reduced equivariant super) \define{Galois} cohomology. These topological orders deserve the name \define{anomalous topological sigma models}. The $(2{+}1)$-dimensional case is briefly addressed in \S\ref{subsec.2+1d}, where we confirm the essence of \cite{10.1093/nsr/nwv077}.
Most of that section is occupied by Remark~\ref{remark.fermioniocWittgroup}, which analyzes a spectrum enhancement of the ``Witt group'' of braided fusion categories.
Then in Theorem~\ref{thm.3+1d}, in \S\ref{subsec.3+1d}, we provide a complete proof of (a mildly corrected version of) the classification of $(3{+}1)$-dimensional topological orders announced in \cite{PhysRevX.8.021074,PhysRevX.9.021005,PhysRevB.100.045105}. 
Specifically, we find that all $(3{+}1)$-dimensional fermionic topological orders are, canonically, anomalous topological sigma models with target a finite groupoid, and the bosonic case is the same together with ``categorified Galois descent data.'' Capping the paper, Remark~\ref{remark.higherdimensions} briefly addresses the extent to which this method of analysis can be pushed to even higher dimensions.

\begin{remark}\label{remark.disclaimer}
  This paper relies heavily on \cite{GJFcond}. In particular, we will make repeated use of Corollary~4.2.4 therein, which provides a way to bootstrap from mild to strong dualizability statements. That Corollary (and indeed most of Sections~4.2 and~4.3 of~\cite{GJFcond}), assumes the existence of a theory of weighted colimits in enriched higher categories. The required properties of such an as-yet-unfinished theory are explicitly stated, and subtle questions are not needed. As such, this author believes that the results in that paper and in this one will compile in any model. But implementing the technical details in a specific model may be a substantial task.
  With luck, this article will stimulate further technical work of higher category theory, and technicians are invited to realize our constructions formally.
\end{remark}

\begin{remark}
  When this paper was near completion, the authors of \cite{KLWZZ} graciously shared a draft of their paper; the two papers were posted on arXiv within days of each other. The two papers have some overlap of ideas and results, and of course the present paper owes many ideas to the work of those authors.
  Notably, 
  Theorem~1.1 of \cite{KLWZZ} is essentially our Definition~\ref{defn.main}, extended to describe topological orders which are enhanced with a ``symmetry group'' 
  described (a la Tannakian duality) by a symmetric fusion $n$-category $\mathcal{R}$. When applied in low dimensions, naturally the classifications essentially match.
  
  The major difference is that, Remark~\ref{remark.disclaimer} notwithstanding, this paper develops the mathematical theory of fusion $n$-categories, which is reviewed but largely taken for granted in \cite{KLWZZ}. In particular, our Theorems~\ref{thm.dualizability}, \ref{thm.invertibility}, and~\ref{thm.anomalousTQFT} establish dualizability and invertibility results in higher Morita categories that justify the attention to fusion categories, and Theorem~\ref{thm.braided} and~\ref{thm.nolines} explain rigorously how it is that $(n{+}1)$-dimensional topological orders which are ``stable'' in the sense that they are robust against small perturbations are fully described by $({<}n)$-categories (without the need to remember the low-codimension extended operators).
\end{remark}

\section{Multifusion $n$-categories}\label{sec.dualizability}

The goal of this section  is to axiomatize the algebras of extended operators in topological orders: we will define ``multifusion $n$-categories'' and their centres. Our definition builds on the definition of multifusion $2$-category by \cite{Reutter2018}.

  We will consistently use the word \define{$n$-category} to mean what is traditionally called \define{weak $n$-category}. Weak $n$-categories have objects through $n$-morphisms. What makes weak $n$-categories ``weak'' is that the composition of $k$-morphisms in a weak $n$-category, for $k<n$, is well-defined only up to a higher coherences, meaning that instead of a single composition, there is really an $(n-k)$-category worth of compositions, and this $(n{-}k)$-category is equivalent to, but usually not equal to, a single point. This distinguishes weak $n$-categories from ``strict $n$-categories,'' which will not be used in this paper. 
    We will sometimes write that two $n$-categorical objects are ``isomorphic'' or even ``equal.''  Of course we never mean equality in the sense of sets, as that notion has no place in $n$-categories. Rather, we mean that they are  equivalent, and that the equivalence is realized via a completely canonical equivalence.
  
  Going in the other direction, weak $n$-categories do not have higher homotopies between $n$-morphisms (the $n$-morphisms form sets, not spaces) and so are not the ``$(\infty,n)$-categories'' studied in much of the contemporary higher category theory literature. But, $(\infty,n)$-category theory does underlie the subject of weak $n$-categories.

  The cleanest definition is to define a \define{weak $0$-category} to be a set, and then to define a \define{weak $n$-category} to be an $(\infty,1)$-category enriched in the $(\infty,1)$-category of weak $(n{-}1)$-categories. 
  One problem with this clean definition is that the complete story of enriched $(\infty,1)$-categories is still being developed \cite{MR3345192}. 
  The definitions are established, but further work is needed in the theory of weighted colimits in order to fully establish \cite[Corollary~4.2.4]{GJFcond}, which as mentioned in Remark~\ref{remark.disclaimer} will be used repeatedly in this paper.

  In order not to distract from the actually important parts of constructions, we will omit from the notation all unitor, associator, and higher coherence data. It's there, but hidden, and the reader may, if they choose, return it. But be careful: the actual data that needs to be returned depends on the model chosen for $n$-categories. An advantage of leaving it out is that we can work model-independently.
  
\begin{remark}  One aspect of the famous Cobordism Hypothesis \cite{BaeDol95,Lur09} is mildly model-dependent.  There is an $(\infty,n)$-category built from framed $n$-manifolds such that functors from it are represented by $n$-dualizable objects. But it may look different in different models, and in particular one could worry that it may deserve the name ``the $(\infty,n)$-category of framed $n$-dimensional cobordisms'' only in some models. 
  Although we will use the Cobordism Hypothesis, we will not require its full strength, and in particular will not bump into this issue: we will work with $n$-dualizable objects (dualizability is a model-independent notion), and write as if we are evaluating the corresponding functorial TQFTs on some cups and caps and spheres, but that is really just a notation for dualizability data.\end{remark}

There are typically no, or very few, nontrivial point, aka vertex, operators in a topological order. (More precisely, the point operators consist just of linear combinations of projections onto the ground states. We will later use the term ``multifusion $n$-category'' when there may be multiple ground states, and ``fusion $n$-category'' for the case when the ground state is nondegenerate.) Thus in order to describe a topological order we must include extended operators. Extended operators in a topological order are also called, variously, ``defects'' and ``excitations.''

We will use the spacetime dimension to count the dimension of operators, so that vertex operators are $0$-dimensional. Line (aka $1$-dimensional) operators are also called \define{anyons}. They are well-known to form a category, with $0$-dimensional junctions between anyons as the morphisms. One may bring anyons together, and since the system is topological, anyons may be ``fused'' without issue; in an $(n{+}1)$-dimensional topological order, anyons furthermore have $n$ ambient dimensions in which they may braid, and so the line operators in a topological order form an ``$n$-monoidal $1$-category.''
 Similarly, surface 
  operators form a $2$-category, with $1$-dimensional junctions as the $1$-morphisms and $0$-dimensional junctions between junctions as the $2$-morphisms; surfaces in $n{+}1$ dimensions have $n{-}1$ ambient dimensions in which to braid, so they form an ``$(n{-}1)$-monoidal $2$-category.''
  Weak $n$-categories satisfy the very useful (Breen--)Baez--Dolan \define{stabilization hypothesis}, which asserts that a $p$-monoidal $q$-category is in fact symmetric (aka $\infty$-monoidal) for $p > q+1$ \cite{SimpsonStabilization}.
  
  \begin{remark}
  A ``$1$-monoidal'' category is merely monoidal; a ``$2$-monoidal'' category is braided monoidal; ``$n$-monoidal'' is the higher generalization. 
At the bottom, $\cC$ is ``$0$-monoidal'' when it is equipped with an object $1 \in \cC$.
An $n$-monoidal category is called ``$n$-tuply monoidal'' in \cite{BaeDol95}, and the contemporary mathematics literature often uses the phrase ``$E_n$-monoidal.''
 The mathematical axioms of ``$p$-monoidal $q$-category'' will not be reviewed here; see for example~\cite{ScheimbauerThesis}. 
 \end{remark}
  
 Any $p$-monoidal $q$-category $\cC$ has an identity object $\mathbf 1$, and 
 its endomorphisms are
 $$ \Omega\cC = \End_{\cC}(\mathbf 1),$$
 which is automatically a $(p{+}1)$-monoidal $(q{-}1)$-category.
 For example, $\Omega\{\text{surface operators}\} = \{\text{line operators}\}$, because a line operator is nothing but a junction from the vacuum surface operator to itself. Continuing up in dimension, we see that an $(n{+}1)$-dimensional topological order has a $1$-monoidal $n$-category $\cA$ of extended operators. The objects of $\cA$ are the codimension-$1$ operators, the 1-morphisms are junctions (of spacetime codimension-$2$) between codimension-$1$ operators, etc. For each $k$, the $(k{+}1)$-monoidal $(n{-}k)$-category $\Omega^k\cA$ consists of the operators in the topological order of codimension $>k$, i.e.\ dimension $\leq n-k$.

\subsection{$n$-categorical condensation}\label{subsec.condensation}

The operators in  an $(n{+}1)$-dimensional topological order are not merely a monoidal $n$-category: in addition to the monoidal $n$-category structure, which encodes how to operators fuse, there are some further ways to produce new operators from old ones. At the bottom, the $0$-dimensional operators form a vector space over $\bC$, and $\cA$ is  a $\bC$-linear $n$-category. (An $n$-category is \define{$\bC$-linear} if the set of all $n$-morphisms with given source and target is in fact a $\bC$-vector space, and if all compositions are $\bC$-linear in each variable.) Higher dimensional operators can also be ``added'' together via a direct sum to produce composite operators, meaning that $\cA$ is \define{additive}. 

Recall that one can detect whether a line operator $X$ is simple or composite by studying the endomorphism algebra $\End(X)$: if $\End(X) \neq \bC$, then it contains a projection $P \in \End(X)$, i.e.\ an endomorphism satisfying $P^2 = P$, and the image of this projection is a direct summand of $X$ and is another line operator. The fact that this image exists is the assertion that the category of anyons is \define{Karoubi (aka idempotent) complete}. 
The notion of ``Karoubi completeness'' also makes sense for $n$-categories, and is due to \cite{GJFcond}, where its tight relationship to (gapped topological) condensation is explored. The main definition $n$-categorifies the notion of projection onto a direct summand:
\begin{definition}
Let $\cC$ be an $n$-category with objects $X,Y \in \cC$. A \define{(categorical) condensation} of $X$ onto $Y$, denoted $X \condense Y$, consists of morphisms $f : X \leftrightarrows Y : g$ together with a condensation $fg \condense \id_Y$. Note that the latter is a condensation in the $(n{-}1)$-category $\End_\cC(Y)$, and so the definition compiles by induction. To begin the induction, equalities are declared to be examples of condensations.
\end{definition}

The correct $n$-categorification of ``idempotent'' is called \define{condensation monad} in \cite{GJFcond}, and
ends up producing a version of ``special Frobenius algebra'': the axiom $P^2 = P$ of a projector gets replaced by the data of a categorical condensation $P^2 \condense P$ which is associative in a suitable sense; unpacked, this provides maps $P^2 \leftrightarrows P$  satisfying associativity, Frobenius, and specialness conditions. 
 The $n$-category $\cC$ is \define{Karoubi complete} when every condensation monad factors through a condensation. (The factorization is automatically unique if it exists.) The $n=2$ version of condensations and condensation monads first appeared (under a different name) in \cite{Reutter2018}.

The construction in~\S2.4 of~\cite{GJFcond} implies that if $\cA$ is the monoidal $n$-category of operators in an $(n{+}1)$-dimensional topological order, then $\cA$ is Karoubi complete.

\begin{remark}\label{remark.physicalcondensationexplainer}
The term \define{monad} refers to an algebra object in the monoidal (higher) category of endomorphisms of some other higher-categorical object. Condensation monads are essentially the ``condensable algebras'' that appear in the modular tensor category literature. The term ``condensable algebra'' is used in a few closely related ways. For example, in Section~2.6 of~\cite{KONG2014436}, Kong identifies condensable algebras $A$ in $(1{+}1)$-dimensional systems (aka \define{1d condensable algebras}) with unital connected symmetric
special  Frobenius algebras; condensation monads (in $(1{+}1)$-dimensions) are almost the same but drop the unitality, connectedness, and symmetry requirements. With or without the extra requirements, to \define{condense} $A$ is to split it in the 2-category of $(1{+}1)$-dimensional systems topological orders: specifically, its condensation is the topological which is microscopically modelled by a network of $A$-lines. 

Dropping the unitality requirement makes no substantive difference by Theorem~3.1.7 of~\cite{GJFcond}. The connectedness condition simply assures that the result of condensation has a nondegenerate ground state; as mentioned already and discussed further in Section~\ref{sec.braided}, we feel like this condition is a useful one to pay attention to, but should not be built into the definitions. Finally, the symmetry requirement of~\cite{KONG2014436} arises physically from requiring that the condensed theory does not require any spin structure or framing in order to be defined. Again, we feel that it is a useful condition to attend to but should not be part of the definition, as a typical microscopic lattice model does come with a framing.

A second usage of the term ``condensable algebra'' appears in the modular tensor category analysis of $(2{+}1)$-dimensional topological orders. Suppose that one has a $(2{+}1)$-dimensional topological order whose line operators form the unitary modular tensor category $\cB$. Then a \define{2d condensable algebras} is a connected commutative separable unital algebras object $A$ in $\cB$, perhaps with an additional symmetry requirement. The separability allows $A$ to be enhanced to a special Frobenius algebra object, and Lemma~3.3.2 of~\cite{GJFcond} implies that the choice of enhancement is irrelevant. The connectedness, unitality, and symmetry requirements play the same role as in the $(1{+}1)$-dimensional case.

If one takes a not-necessarily-commutative special Frobenius algebra $A$ in $\cB$, its condensation is its splitting in the 2-category of surface and line operators. Explicitly, the condensation of $A$ is the surface operator which is microscopically modelled by a network of $A$-lines. The commutativity requirement on $A$ supplies extra structure to the condensate surface operator: specifically, it makes that surface operator into a condensation algebra object in the monoidal 2-category of surface operators. In other words, our $(2{+}1)$-dimensional topological order, imagined as an object in some 3-category of all $(2{+}1)$-dimensional topological orders, would support a condensation monad. The image aka condensate of that condensation monad would be the new phase produced from condensing $A$ in the style of \cite{KONG2014436}.
  
  In general dimension, the translation between physical and categorical condensation is as follows. Suppose $X$ is some $k$-morphism in $\cA$, the $n$-category of all operators, and $P \in \End(X)$ is a condensation monad. The physical condensate $Y$ is built by decorating an $X$ defect with a microscopic network of condensation monad data. In the other direction, given a physical condensation procedure that produces $Y$ as a condensate of $X$, one considers a $Y$-defect with a collection of small bubbles of $X$. By expanding those bubbles, one builds in this way a network of data on $X$, and these data precisely compile into a condensation monad. This ``expand the bubbles'' idea is what motivates the construction in~\S2.4 of~\cite{GJFcond}; related discussion is in \S3.3~of~\cite{KLWZZ} and  Footnote~4 of \cite{KONG2014436}.
\end{remark}

We will henceforth write $n\cat{KarCat}_\bC$ for the $(n{+}1)$-category of $\bC$-linear additive and Karoubi complete $n$-categories (and arbitrary $\bC$-linear functors).
The construction $\cC \leadsto \Omega\cC$ taking a $k$-monoidal $m$-category to a $(k{+}1)$-monoidal $(m{-}1)$-category has an adjoint in the Karoubi-complete world. Denoted $\cC \leadsto \Sigma\cC$, it takes a $k$-monoidal $m$-category $\cC$ first to its ``one-point delooping'' $\rB\cC$, which is the $(k{-}1)$-monoidal $(m{+}1)$-category with one object ``$\bullet$'' and $\End(\bullet) = \cC$, and then Karoubi completes:
$$\Sigma \cC := \Kar(\rB\cC).$$
The Karoubi completion $\Kar(-)$ is described explicitly in \cite{GJFcond}. For 1-categories it is the familiar idempotent completion. 
 Under the hypothesis that all objects in $\cC$ have duals, Theorem~3.3.3 of \cite{GJFcond} asserts that $\Sigma\cC$ is equivalent to the $(m{+}1)$-category of separable associative algebra objects internal to $\cC$, with 1-morphisms given by internal-to-$\cC$ bimodules. As a matter of convention, we also declare:
$$\Sigma\bC := \Vect^\fd_\bC.$$
This convention is justified in Example 4.3.6 of \cite{GJFcond}.

We will make repeated use of the following technical result, which is essentially the ``Categorical Schur's Lemma'' \cite[Proposition 1.2.19]{Reutter2018}:
\begin{lemma}\label{lemma.technical}
  Suppose that $\cS$ is a Karoubi-complete, additive, and $\bC$-linear $n$-category, pointed by an object $\mathbf 1 \in \cS$. Define $\Omega\cS = \End_\cS(\mathbf 1)$, $\Omega^2\cS = \End_{\Omega\cS}(\id_{\mathbf 1})$, and so on, so that $\Omega^k\cS$ is a $k$-monoidal $(n{-}k)$-category. Suppose that the $\bC$-linear $1$-category $\Omega^{n-1}\cS$ is semisimple and that $\Omega^n\cS = \bC$. Then every nonzero fully adjunctible $1$-morphism $f : X \to \mathbf 1$ in $\cS$ extends to a condensation $X \condense \mathbf 1$. 
\end{lemma}
Note that semisimplicity of $\Omega^{n-1}\cS$ follows, for example, from dualizability of $\Omega^{n-1}\cS$ in $\cat{KarCat}_\bC$, which follows, for example, from dualizability of $\cS$ in $n\cat{KarCat}_\bC$.
\begin{proof}
  When $n=1$, we have a semisimple $\bC$-linear category $\cS$ with an object $\mathbf 1 \in \cS$ which is simple since $\Omega\cS = \End_\cS(\mathbf 1) = \bC$. Then any nonzero $f : X \to \mathbf 1$ has a splitting $g : \mathbf 1 \to X$, and together $(f,g)$ are a condensation $X \condense \mathbf 1$.
  
  When $n > 1$, choose $g : \mathbf 1 \to X$ to be the right adjoint of $f$, and let $\phi : fg \Rightarrow \id_{\mathbf 1}$ be the counit of the adjunction. Then $\phi$ is a $1$-morphism in the $(n{-}1)$-category $\Omega\cS$, which satisfies all conditions of the Theorem: full adjunctibility of $\phi$ is part of fully adjunctibility of $f$, and nonzeroness of $\phi$ is equivalent to nonzeroness of $f$. So $\phi$ extends to a condensation $fg \condense \id_{\mathbf 1}$ by induction, and so we have built a condensation $X \condense \mathbf 1$.
\end{proof}

\subsection{$\cat{Mor}_1(n\cat{KarCat}_\bC)$}\label{subsec.mor1}

We thus find ourselves interested in monoidal objects of $n\cat{KarCat}_\bC$. In order to cleanly study such objects, we quote without proof two facts about the $(n{+}1)$-category $n\cat{KarCat}_\bC$. 
The first fact is that $n\cat{KarCat}_\bC$ is symmetric monoidal under a \define{Karoubi-completed tensor product} $\boxtimes$. To define it, assume by induction that the Karoubi-completed tensor product is defined on $(n{-}1)\cat{KarCat}_\bC$. Now define, for $\cA,\cB \in n\cat{KarCat}_\bC$, their \define{na\"ive tensor product} $\cA \otimes \cB$ to be the $n$-category with object set $$\ob(\cA \otimes \cB) = \ob(\cA) \times \ob(\cB)$$ and morphism spaces \begin{multline*}\hom_{\cA \otimes \cB}((A_1,B_1),(A_2,B_2)) \\ = \hom_\cA(A_1,A_2) \boxtimes \hom_\cB(B_1,B_2).\end{multline*} (By definition, $\hom_\cA(A_1,A_2)$ and $\hom_\cB(B_1,B_2)$ are each Karoubi-complete $(n{-}1)$-categories, since $\cA$ and $\cB$ are Karoubi-complete $n$-categories.)
 The tensor product $\cA \boxtimes \cB$ is defined to be the Karoubi completion of the na\"ive tensor product:
$$ \cA \boxtimes \cB = \Kar(\cA \otimes \cB).$$
That this definition compiles uses implicitly Theorem~2.3.10 of \cite{GJFcond}, which asserts that $\Kar(-)$ takes products to products; in turn, that this definition compiles is used in order to define $\Kar(-)$ for $(n{+}1)$-categories. The unit object in $n\cat{KarCat}_\bC$ is $\Sigma^n\bC = \Sigma^{n-1}\cat{Vec}^\fd_\bC$.

The second fact we will need about $n\cat{KarCat}_\bC$ is that it is closed under colimits, and that $\boxtimes$ distributes over colimits. We will not say much about this fact, simply observing that it follows from the definition $\boxtimes := \Kar(\otimes)$, together with the 
adjunction between $\Kar(-)$
and the forgetful functor from $n\cat{KarCat}_\bC$ to the $(n{+}1)$-category $n\cat{Cat}_\bC$ of all $\bC$-linear $n$-categories. With these two facts in place, the construction from \S8 of \cite{JFS} builds a symmetric monoidal $(n{+}2)$-category $\cat{Mor}_1(n\cat{KarCat}_\bC)$ whose objects are the monoid objects in $n\cat{KarCat}_\bC$ and whose $1$-morphisms are the (unpointed) bimodule objects. 
This category would be called ``unpointed $\cat{Alg}_1^{\mathrm{strong}}(n\cat{KarCat}_\bC)$'' in the language of \cite{JFS}. We will call it ``$\cat{Mor}_1$'' because equivalence therein is \define{Morita equivalence}.

\begin{remark}\label{remark.stackingdescription}
  So as not to stray too far from the physics, let us unpack these constructions in the case of topological orders. Suppose $\cA$ and $\cB$ are monoid objects in $n\cat{KarCat}_\bC$ describing $(n{+}1)$-dimensional topological orders. These topological orders may be layered (aka stacked), by placing the orders on adjacent sheets separated by an infinitely-good insulator; write $\cC$ for the corresponding monoidal $n$-category.
  Operators in $\cA$ and $\cB$ determine operators in the layered system, and so there is a functor $\cA \otimes \cB \to \cC$. But $\cC$ is Karoubi-complete, and so this functor extends to a functor $\cA \boxtimes \cB \to \cC$ (which we will later assume is an equivalence; see \S\ref{sec.justification}).
  
  The physical meaning of $\cat{Mor}_1(n\cat{KarCat}_\bC)$ is the following. Suppose $\cA$ and $\cB$ describe $(n{+}1)$-dimensional topological orders, and choose a topological interface between them. That interface is $n$-spacetime-dimensional, and so supports a monoidal $(n{-}1)$-category $\cM$ of operators. These three categories are compatible. For example, operators in $\cA$ or $\cB$ of codimension $\geq 2$ may be brought to the interface to define operators in $\cM$; operators of codimension-$1$, i.e.\ objects, in $\cA$ or $\cB$ may be fused with $\cM$ to define new interfaces. Consider (the Karoubi completion of) the $n$-category of all $\cA$-$\cB$ interfaces that can be built from $\cM$ in this way. That $n$-category is naturally an $\cA$-$\cB$ bimodule, i.e.\ a $1$-morphism in $\cat{Mor}_1(n\cat{KarCat}_\bC)$. (But note that $\cat{Mor}_1(n\cat{KarCat}_\bC)$ contains other $1$-morphisms that do not arise in this way.)
\end{remark}

\subsection{A dualizability result}\label{subsec.dualizability}

A $1$-morphism $f : X \to Y$ in an $n$-category \define{has adjoints} if there exist $1$-morphisms $f^R, f^L : Y \rightrightarrows X$ and evaluation and coevaluation $2$-morphisms $\id_X \to f^R f$, $ff^R \to \id_Y$, $\id_Y \to ff^L$, $f^Lf \to \id_X$ such that the various compositions $f \to ff^L f \to f$, $f \to ff^R f \to f$, $f^L \to f^L f f^L \to f^L$, and $f^R \to f^R f f^R \to f^R$ that can be made from evaluation and coevaluation are all equivalent to identities. A morphism is \define{adjunctible} if it has adjoints and its adjoints have adjoints, ad infinitum. An object $X \in \cC$ in a monoidal $n$-category is \define{$1$-dualizable} if it is adjunctible as a $1$-morphism in the $(n{+}1)$-category $\rB\cC$. In many situations there is an isomorphism $f^R \cong f^L$, and so adjunctibility is not infinitely much data. For instance, if $\cC$ is symmetric monoidal, then $X^R \cong X^L$ for any $1$-dualizable object $X \in \cC$. If $\cC$ is a symmetric monoidal $n$-category, then $X \in \cC$ is \define{$k$-dualizable}  if it is dualizable, and the evaluation and coevaluation morphisms witnessing dualizability are adjunctible, and their evaluation and coevaluation morphisms are adjunctible, and so on, up to dimension $k$, and \define{fully dualizable} if it is $n$-dualizable. (The only way an object in a symmetric monoidal $n$-category can be $(>n)$-dualizable is if it is invertible.)

It is a well-known folk theorem that, for any reasonable symmetric monoidal higher category $\cS$, {every} object of $\cat{Mor}_1(\cS)$ is $1$-dualizable; c.f.~\cite{Schommer-Pries:thesis,MR3221292,ScheimbauerThesis,GwSch18}. (Sufficient conditions on $\cS$ are that it be
 what in \cite{JFS} is called ``$\otimes$-GR-cocomplete''; indeed, $\otimes$-GR-cocompleteness is required to construct $\cat{Mor}_1(\cS)$, and simply asserts the existence of certain colimits in $\cS$, which are preserved by the monoidal structure.)
 Applied to the $(n{+}1)$-category $\cS = n\cat{KarCat}_\bC$, we learn that if $\cA$ is the monoidal $n$-category of operators in an $(n{+}1)$-dimensional topological order, then $\cA$ is $1$-dualizable as an object in $\cat{Mor}_1(n\cat{KarCat}_\bC)$. We will argue in \S\ref{sec.justification} that it is (at least) $2$-dualizable. But:

\begin{theorem}\label{thm.dualizability}
   If $\cA \in \Mor_1(n\cat{KarCat}_\bC)$ is $2$-dualizable, then it is fully, i.e.\ $(n{+}2)$-dualizable, dualizable.
\end{theorem}

We will give two proofs of this fundamental result, because they illustrate different features of $\Mor_1(n\cat{KarCat}_\bC)$.

\begin{proof}[First proof of Theorem~\ref{thm.dualizability}]
 When $n=0$, so that $n\cat{KarCat}_\bC = \Vect_\bC$, there is nothing to prove. Recall the standard fact that an algebra $A \in \Mor_1(\Vect)$ is $2$-dualizable if and only if it is both:
 \begin{itemize}
   \item \define{proper}, aka finite-dimensional, i.e.\ the underlying vector space of $A$ is $1$-dualizable;
   \item \define{smooth}, aka separable, i.e.\ the multiplication map $m: A \otimes A \to A$ splits as a map of $A$-bimodules.
 \end{itemize}
We will prove a version of this fact in higher categories. To begin, recall that  $\cA$ is automatically $1$-dualizable in $\Mor_1(n\cat{KarCat}_\bC)$, with dual $\cA^\op$, and so the question of $2$-dualizability amounts to asking whether the bimodule $_{\cA^e}\cA_{\mathbf 1}$ has both adjoints, where $\cA^e = \cA \boxtimes \cA^\op$ is the \define{enveloping algebra} of $\cA$ and $\mathbf 1 = \Sigma^n\bC \in \Mor_1(n\cat{KarCat}_\bC)$ is the unit object.

Suppose that $\cS$ is any reasonable symmetric monoidal higher category, with tensor product $\otimes$ and unit~$\mathbf{1}$, and that $P,Q \in \Mor_1(\cS)$ are unital associative algebra objects in $\cS$, and that $_P M_Q$ is a bimodule between them, which has both adjoints. One of the adjoints of $M$ (whether left or right depends on convention) will be a bimodule $_Q N _P$ equipped with maps
\begin{align*} \epsilon &:\, {_Q N _P}\, \underset P\otimes\, {_P M _Q} \to {_Q Q _Q}, \\
\eta &:\, {_P P _P} \to {_P M _Q} \,\underset Q\otimes\, {_Q N _P}. \end{align*}
Precompose $\epsilon$ with the quotient map $$N \otimes M \to N \,\underset P\otimes\, M,$$ and evaluate $\eta$ on the identity element $1 \in P$. The result are maps
\begin{align*}
\epsilon' & :\, {_Q N} \,\otimes\, {M_Q} \to {_Q Q_Q}, \\
\eta' &:\, \mathbf{1} \to M \underset Q\otimes N, \end{align*}
which witness ${_Q N}$ as the adjoint to ${M_Q}$. Moreover, these maps identify
$$ N \cong \hom_Q({M_Q}, {Q_Q}).$$
The naturality of $\epsilon',\eta'$ implies that this identification is in fact an isomorphism of $Q$--$P$ bimodules.

Applying this to ${_P M_Q} = {_{\cA^e}\cA_{\mathbf 1}}$ implies in particular that the underlying $n$-category of $\cA$ is $1$-dualizable in $n\cat{KarCat}_\bC$. But Corollary 4.2.4 of \cite{GJFcond} then implies that it is fully-dualizable in $n\cat{KarCat}_\bC$, and is an object of the subcategory $\Sigma^n\Vect_\bC$. This is the sense in which $\cA$ is \define{proper}.

We must now investigate adjunctibility on the other side, i.e.\ with the roles of $P$ and $Q$ reversed. Theorem 4.2.2 of \cite{GJFcond}, applied to $\cC = \rB (\cA^e)$, says that $\cA$ is adjunctible as an $\cA^e$-module if and only if $\cA \in \Sigma(\cA^e)$, meaning that there is a categorical condensation $\cA^e \condense \cA$ in the $(n{+}1)$-category of $\cA^e$-modules. Indeed, the proof of that theorem provides an algorithm to find such a condensation from a presentation of $\cA$ as a $\cA^e$-module. Note that the multiplication $m : \cA^e \to \cA$ is a map of $\cA^e$-modules; it witnesses $\cA$ as a cyclic module (i.e.\ a module with only one generator), and can be extended to a presentation. Applied to this presentation, the proof then constructs an $\cA^e$-linear map $\Delta : \cA \to \cA^e$ ``splitting'' the multiplication $m$ in the sense that $m \circ \Delta : \cA \to \cA$ can be further condensed onto the identity. This is the sense in which $\cA$ is \define{smooth}.

But smoothness of $\cA$ is nothing but the statement that $\cA$ is a ``condensation algebra'' in the sense of \cite{GJFcond}, and in particular (since it is also proper) an object in $\Sigma(\Sigma^n\Vect_\bC) \subset \Mor_1(\Sigma^n\Vect_\bC)$. Thus $\cA$ is fully dualizable by Theorem 4.1.1 of \cite{GJFcond}.
\end{proof}

\begin{proof}[Second proof of Theorem~\ref{thm.dualizability}]
There is not a functor $\cat{Mor}_1(n\cat{KarCat}_\bC) \to (n{+}1)\cat{KarCat}_\bC$ extending the construction $\cA \mapsto \Sigma\cA$. The problem is to assign the values on $1$-morphisms.
 A functor $\Sigma\cA \to \Sigma\cB$ of $(n{+}1)$-categories is determined by where it takes the object $\bullet \in \rB\cA \subset \Sigma\cA$. We would like to assign to an $\cA$-$\cB$-bimodule $\cM$ the functor that takes $\bullet$ to $\cM$-as-a-$\cB$-module. But this is an object of $\Sigma\cB \subset \Mod(\cB)$ if and only if it is adjunctible on the $\cB$-side (c.f.\ Corollary 4.2.4 of \cite{GJFcond}).
 
 The upshot is that $\Sigma(-)$ does extend to a faithful functor to $(n{+}1)\cat{KarCat}_\bC$ from the sub-$(n{+}2)$-category of $\cat{Mor}_1(n\cat{KarCat}_\bC)$ on the same objects, but with only the adjunctible $1$-morphisms. In particular, if $\cA$ is at least $2$-dualizable in $\cat{Mor}_1(n\cat{KarCat}_\bC)$, then $\Sigma\cA$ is at least $2$-dualizable in $(n{+}1)\cat{KarCat}_\bC$. But then $\Sigma\cA$ is fully dualizable in $(n{+}1)\cat{KarCat}_\bC$ by Corollary 4.2.4 of~\cite{GJFcond}, and hence fully dualizable in $\cat{Mor}_1(n\cat{KarCat}_\bC)$.
\end{proof}

\begin{remark}\label{remark.bestiary}
The following remark about Theorem~\ref{thm.dualizability} is based on a suggestion by C.\ Scheimbauer. 
 By definition, a symmetric monoidal $\bC$-linear $(n{+}1)$-category $\cC$  satisfies the \define{Bestiary Hypothesis} if its fully dualizable subcategory $\cC^\fd$ is equivalent to $\Sigma^{n+1}\bC$. This happens precisely when $\cC$ is additive and Karoubi complete, $\Omega\cC$ satisfies the Bestiary Hypothesis, and furthermore all objects of $\cC^\fd$ are condensates of the identity object.
 
 Scheimbauer has suggested that the Bestiary Hypothesis be used as a desideratum when evaluation a proposal for ``the $n$-category of $n$-vector spaces.'' The name comes from the appendix of \cite{BDSPV}, which establishes the Bestiary Hypothesis for many (but not all) of the $2$-categories that have been proposed in the literature as ``the $2$-category of $2$-vector spaces.''  Corollary 4.2.3 of \cite{GJFcond}, applied to $\cC = \Sigma^n\bC$, can be summarized by saying that $n\cat{KarCat}_\bC$ satisfies the Bestiary Hypothesis.
 
 The proofs of Theorem~\ref{thm.dualizability} amount to the fact that $\cat{Mor}_1(n\cat{KarCat}_\bC)$ also satisfies the Bestiary Hypothesis. Indeed, the first proof explicitly realizes dualizable objects in $\cat{Mor}_1(n\cat{KarCat}_\bC)$ as condensates of the identity object. The second proof argues more abstractly that the fully dualizable sub-$(n{+}2)$-categories of $\cat{Mor}_1(n\cat{KarCat}_\bC)$ and of $(n{+}1)\cat{KarCat}_\bC$ agree.
 
 Theorem~\ref{thm.dualizability} should generalize to the statement that if an $(n{+}1)$-category $\cS$ satisfies the Bestiary Hypothesis, then so does the $(n{+}2)$-category $\Mor_1(\cS)$. The motivated reader is invited to publish the details.
\end{remark}

When $n=1$ we recover some familiar category theory. Recall that a \define{multifusion (1-)category} is a $\bC$-linear, additive, and Karoubi-complete monoidal $1$-category which is semisimple with finitely many isomorphism classes of simple objects, and such that each object $X \in \cA$ is dualizable.

\begin{corollary}\label{cor.multifusion}
  $\cA \in \cat{Mor}_1(\cat{KarCat}_\bC)$ is fully dualizable if and only if it is a multifusion category.
\end{corollary}

\begin{proof}
 Since $\bC$ is of characteristic zero, every multifusion category over $\bC$ is separable, and full dualizability follows from \cite{DSPS}.
 
 In the other direction, first note that in order for $\cA$ to be ``proper'' in the sense that the underlying object of $\cA$ is dualizable in $\cat{KarCat}_\bC$, it must be semisimple with finitely many isomorphism classes of simples \cite{MR1670122}. As for smoothness,
  Theorem 3.1.7 of \cite{GJFcond} implies that the splitting $\Delta$ of $m : \cA \boxtimes \cA \to \cA$ in the  first proof of Theorem~\ref{thm.dualizability} may be chosen to be the right adjoint $1$-morphism $\Delta = m^R$. That $m^R$ is a map of $\cA$-bimodules is equivalent to the statement that every object in $\cA$ has duals; c.f.\ Definition-Proposition 1.3 of \cite{1804.07538} or Appendix D of \cite{MR3381473}.
\end{proof}

This establishes Conjecture 3.3.5 of \cite{GJFcond}. It also justifies:
\begin{definition}\label{defn.multifusion}
  A monoidal $n$-category is \define{multifusion} if it is additive, $\bC$-linear, Karoubi-complete, and fully dualizable in $\cat{Mor}_1(\cat{KarCat}_\bC)$.
\end{definition}

\begin{remark}\label{remark.multifusion}
Corollary~\ref{cor.multifusion} and Definition~\ref{defn.multifusion} obscure an important subtlety. Over fields of positive characteristic, there are \define{inseparable} multifusion categories,
and what we are calling ``multifusion $n$-categories'' should really be called ``separable multifusion $n$-categories.''

For an inseparable fusion 1-category $\cA$, the multiplication $m : \cA \boxtimes \cA \to \cA$ and its right adjoint $m^R$ are not the components of a condensation $\cA \boxtimes \cA \condense \cA$, and $\cA$ is not $2$-dualizable in $\cat{Mor}_1(\cat{KarCat}_\bC)$. However, such an $\cA$ is $2$-dualizable in a closely related $3$-category $\cat{Mor}_1(\cat{Rex}_\bC)$, which is the $3$-category used in \cite{DSPS}. The difference between $\cat{Rex}_\bC$ and $\cat{KarCat}_\bC$ is that in the former, but not the latter, all categories are required to have cokernels, and all functors are required to preserve cokernels. The upshot is that the monoidal structure on $\cat{Rex}_\bC$ is not merely the Karoubi completion of the na\"ive tensor product, but rather completes also for cokernels. This change affects, for example, the composition of bimodules, and so throws off the ``smoothness'' part of the first proof of Theorem~\ref{thm.dualizability}. From the perspective of the second proof, it would require working with a functor from $\cat{Mor}_1(\cat{Rex}_\bC)$ to some ``$2\cat{Rex}_\bC$,'' and we would not have access to Theorem 4.1.1 of \cite{GJFcond}.
\end{remark}

\begin{lemma}\label{lemma.rigid}
  A multifusion $n$-category $\cA$ is \define{rigid}: every object admits left and right duals, and every $k$-morphism for $k < n$ admits left and right adjoints
\end{lemma}

\begin{proof}
  Since $\cA$ is unital, it arises as the endomorphisms of its own rank-1 free module. This can be said more categorically: consider the bimodule $_\cA \cA_{\mathbf 1}$ as a 1-morphism in the $(n{+}2)$-category $\Mor_1(n\cat{KarCat}_\bC)$; then $\cA \cong \End(_\cA \cA_{\mathbf 1})$.
  
  For any monoidal $n$-category, $_\cA \cA_{\mathbf 1}$ automatically has an adjoint over $\cA$. The ``properness'' identified in the first proof of Theorem~\ref{thm.dualizability} says that it is also adjunctible over  $\mathbf 1$.  
   Thus, like $\cA$ itself, the rank-1 free module $_\cA \cA_{\mathbf 1}$ lives inside the fully-dualizable sub-$(n{+}2)$-category $\Sigma^{n+2}\bC \subset \Mor_1(n\cat{KarCat}_\bC)$.
   
   But then $\cA$ is the $n$-category of endomorphisms of a 1-morphism in an $(n{+}2)$-category with all duals.
\end{proof}

\begin{proposition}\label{prop.subfusion}
  If $\cA$ is multifusion in the sense of Definition~\ref{defn.multifusion}, and if $\cA' \subset \cA$ is a Karoubi complete and additive full rigid monoidal subcategory (i.e.\ $\cA'$ is closed under multiplication and taking duals), then $\cA'$ is multifusion in the sense of Definition~\ref{defn.multifusion}.
\end{proposition}

\begin{proof}
  Independently of the monoidal structure, suppose that $\cA \in \Sigma^{n+1}\bC \subset n\cat{KarCat}_\bC$ is dualizable. Then endomorphism categories in $\cA$ are also dualizable. It follows that for every object $X \in \cA$, the $\bC$-algebra of endo-$n$-morphisms $\Omega^{n-1} \End_\cA(X)$ is finite-dimensional and semisimple, and so $X$ decomposes (canonically, provided $n\geq 2$, since then $\Omega^{n-1} \End_\cA(X)$ is commutative) as a direct sum of simple objects. Together with Lemma~\ref{lemma.technical}, this implies that any Karoubi complete and additive full sub-$n$-category $\cA' \subset \cA$ is a direct summand (compare \cite[Remark~1.2.23]{Reutter2018}), and hence dualizable. In the case of the Proposition, this establishes that $\cA'$ is proper.
  
  The more interesting part is to establish that $\cA'$ is smooth, i.e.\ that the multiplication map $m' : \cA' \boxtimes \cA' \to \cA'$ extends to an $\cA'$-bilinear condensation. Of course, this $m'$ is the restriction of the multiplication map $m : \cA \boxtimes \cA \to \cA$. By Lemma~\ref{lemma.rigid}, $\cA$ is rigid, and so, as in the proof of Corollary~\ref{cor.multifusion}, it has an $\cA$-bilinear adjoint $m^R : \cA \to \cA \boxtimes \cA$, and we may take this adjoint as part of the condensation $\cA \boxtimes \cA \condense \cA$. On the other hand, since $\cA'$ by assumption also has duals, $m'$ has an $\cA'$-bilinear adjoint $(m')^R$, which (since $\cA' \subset \cA$ is a direct summand) must be simply the restriction of $m^R$.
  
  Finally, $m$ and $m^R$ are the downward and upward arrows in a condensation $\cA \boxtimes \cA \condense \cA$, so that there is a condensation $mm^R \condense \id_\cA$, which consists of some natural transformations between endofunctors of $\cA$ (and higher natural transformations between those natural transformations). Restrict these data to $\cA'$. This is possible because $\cA'$ is a full subcategory of $\cA$, and the functors themselves do factor through $\cA'$. It follows that $m'$ and $(m')^R$ extend to a condensation $\cA' \boxtimes \cA' \condense \cA'$. It is $\cA'$-bilinear simply by restricting the bilinearity data from $\cA$ (and using the fullness of $\cA' \subset \cA$ to confirm that the restrictions exist).
\end{proof}

\subsection{Centres and invertibility}\label{sec.centre}

Any monoidal $n$-category $\cA$ has a \define{(Drinfeld) centre} $\cZ(\cA)$, which measures the ability of objects in $\cA$ to commute past other objects. 
Recall from the first proof of Theorem~\ref{thm.dualizability} the \define{enveloping algebra} $\cA^e = \cA \boxtimes \cA^\op$, which acts canonically on $\cA$. The shortest definition of $\cZ(\cA)$ is:
$$ \cZ(\cA) = \End_{\Mod(\cA^e)}(\cA).$$

\begin{remark}
  The universal property of Karoubi completion guarantees that Karoubi-complete $\cA^e$-modules are equivalent to Karoubi-complete modules for the ``na\"ive enveloping algebra'' $\cA \otimes \cA^\op$. In particular, $\cZ(\cA)$ doesn't change if computed in the Karoubi-completed or incomplete worlds. One can see this another way. Endomorphism categories can be computed as limits. 
  Write $n\cat{Cat}$ for the $(n{+}1)$-category of all $n$-categories.
  Since the forgetful functor $n\cat{KarCat}_\bC \subset n\cat{Cat}$ is a right adjoint, it preserves limits, and so if $\cA$ is a monoidal object in $n\cat{KarCat}_\bC$, then $\cZ(\cA)$ can be computed in plain $n\cat{Cat}$ and the result will be automatically Karoubi complete, additive, and linear.
\end{remark}

\begin{remark}\label{remark.centres}
Suppose that $\cA$ is in fact multifusion. 
Together with the Cobordism Hypothesis,
Theorem~\ref{thm.dualizability} implies that there is a (unique) framed fully extended $(n{+}2)$-dimensional TQFT with values in $\cat{Mor}_1(n\cat{KarCat}_\bC)$ whose value on a point is $\cA$. We will write this TQFT as $M \mapsto \int_M \cA$. Let $S^1_b = \partial D^2$ denote the $1$-dimensional circle with boundary framing. There is a reasonably well-known identification
$$ \cZ(\cA) \cong \int_{S^1_{b}} \cA.$$
Indeed, consider $\cA$ as a $\cA^e$-module, and its ``$\cA^e$-linear dual'' module $\cA^* :=  \hom_{\cA^e}(\cA, \cA^e)$.  Then $\int_{S^1_{b}}$ can be presented as the tensor product $\cA^* \otimes_{\cA^e} \cA$. But the dualizability of $\cA \in \Mor_1(n\cat{KarCat}_\bC)$ allows us to identify this tensor product with $\hom_{\cA^e}(\cA, \cA) = \cZ(\cA)$. 
\end{remark}

Unpacking the definition,
an object of $\cZ(\cA)$ is an object $X \in \ob(\cA)$ together with a natural-in-$Y$ isomorphism $x_Y : X \otimes Y \isom Y \otimes X$, which must furthermore be compatible with multiplication in the $Y$-variable, so that for example $x_{Y_1 \otimes Y_2} \cong (\id_{Y_1} \otimes x_{Y_2}) \circ (x_{Y_1} \otimes \id_{Y_2})$.  When $\cA$ is a $(\geq 2)$-category, ``naturality in $Y$'' is extra data on the isomorphism $x_{(-)}$, not just a property. For example, given a $1$-morphism $y : Y \to Y'$ in~$\cA$, the naturality data includes the data of a $2$-isomorphism interpolating between $x_{Y'} \circ (\id_X \otimes y)$ and $(y \otimes \id_{X}) \circ x_{Y}$; similarly, for each $2$-morphism, ``naturality in $Y$'' requires a $3$-isomorphism; etc. Since the objects of $\cZ(\cA)$ are pairs $(X, x_{(-)})$, similarly the morphisms are pairs. For example, a $1$-morphism $(X, x_{(-)}) \to (X', x'_{(-)})$ consists of a $1$-morphism $\Xi : X \to X'$ together with an isomorphism $\xi_{(-)}$ between the natural isomorphisms $(\id_{(-)} \otimes \Xi) \circ x_{(-)}$ and $x'_{(-)} \circ (\Xi \otimes \id_{(-)})$. Again $\xi_{(-)}$ consists of much data: for each $k$-morphism $y$, we ask for a $(k{+}2)$-isomorphism $\xi_y$. In general, an $m$-morphism in $\cZ(\cA)$ consists of an $m$-morphism in $\cA$  and, for each $k$-morphism, an $(m{+}k{+}1)$-isomorphism.

As an example of an object in $\cZ(\cA)$, we may take $X$ to be the unit object $\mathbf{1}_\cA$ and $x_{Y}$ to be the identity map $\id_Y : \mathbf{1}_\cA \otimes Y \isom Y \otimes \mathbf{1}_\cA$ (or, more precisely, an appropriate composition of unitor and associator data). We will be interested in the case where this is the ``only'' object of $\cZ(\cA)$. It cannot be truly the only object, since $\cZ(\cA)$ is additive and Karoubi complete: by taking direct sums and splitting higher categorical idempotents, the trivial object $(\mathbf{1}_\cA, \id_{(-)}) \in \cZ(\cA)$ closes to a copy of $\Sigma^n \bC$ inside $\cZ(\cA)$. This copy of $\Sigma^n\bC$ should be thought of as the ``multiples of the identity.''

\begin{definition}
  A Karoubi complete, additive, and $\bC$-linear monoidal $n$-category $\cA$ \define{has trivial centre} when the map $\Sigma^n\bC \to \cZ(\cA)$ picked out by the unit object is an equivalence.
\end{definition}

\begin{theorem}\label{thm.invertibility}
  An object of $\cat{Mor}_1(n\cat{KarCat}_\bC)$ is invertible if and only if it is multifusion with trivial centre.
\end{theorem}

Related results can be found in \cite{BJSS2020}, which appeared on arXiv after the first version of this paper was posted.

\begin{proof}
The $n=0$ case is the well-known fact that the Morita-invertible algebras, over any base field, are precisely the finite-dimensional separable algebras with trivial centre. We henceforth assume $n\geq 1$.

  Suppose $\cA \in \cat{Mor}_1(n\cat{KarCat}_\bC)$ is invertible. Then it is dualizable and hence multifusion.
  Write $M \mapsto \int_M \cA$ for the corresponding $(n{+}2)$-dimensional fully extended TQFT with values in $\cat{Mor}_1(n\cat{KarCat}_\bC)$. Since $\cA$ is invertible, this TQFT takes invertible values. In particular, the map
  $$ \Sigma^n\bC = \int_{\emptyset} \cA \overset{\int_{D^2}\cA}\longto \int_{S^1_b}\cA = \cZ(\cA)$$
  is an equivalence. The first equality is from the definition of TQFT, and the second is from Remark~\ref{remark.centres}. 
  This establishes the ``only if'' direction of the Theorem.
  
  For the ``if'' direction, since $\cA$ is multifusion, it is (at least) $(n{+}2)$-dualizable in $\Mor_1(n\cat{KarCat}_\bC)$, and, since $n+2 \geq 3$, we can apply the main result of \cite{MR3811769}, which implies that $\cA$ is invertible if $\int_{S^1_{b}} \cA$ is.
  But $\int_{S^1_b}\cA = \cZ(\cA) = \Sigma^n\bC$ is trivial, and so invertible.
In fact, this application is substantially easier than the full strength of \cite{MR3811769}, and is essentially due to unpublished work by D.\ Freed and C.\ Teleman described in \cite{FreedAspects}. Careful analysis of \cite{MR3811769} shows also that this application also does not require the full strength of the Cobordism Hypothesis: one can express the entire argument simply in terms of dualizability data.
\end{proof}

\section{Justifying our definition of topological order} \label{sec.justification}

The goal of this section is to justify Definition~\ref{defn.main}, now that we have defined the constituent terms. Unpacked, that definition identifies $(n{+}1)$-dimensional topological orders with monoidal $n$-categories which are additive, $\bC$-linear, Karoubi complete, dualizable, and have trivial centre. We first list carefully the physical ans\"atze contributing to these conditions, arguing for their correctness but also commenting on the challenges that each assumption might face. We then argue that these conditions are complete, in the sense that no further conditions are needed in order to determine a topological order; this completeness justifies the use of the term ``ansatz'' rather than merely ``assumption.''

\subsection{Correctness}\label{subsec.correct}

Definition~\ref{defn.main} is based on the following physical ans\"atze:
\begin{enumerate}\setcounter{enumi}{-1}
  \item \label{ansatz.heisenberg} A {topological order} determines and is determined by its ``algebra'' $\cA$ of low-energy topological operators.
  \item \label{ansatz.category} The ``algebra'' $\cA$ in an $(n{+}1)$-dimensional topological order is a monoidal $n$-category. The $(k{+}1)$-monoidal $(n{-}k)$-category of $(n{-}k)$-dimensional operators is $\Omega^k\cA$. Furthermore, $\cA$ is $\bC$-linear, additive, and Karoubi complete.
  \item \label{ansatz.stacking} The stacking of topological orders is described by the Karoubi-completed tensor product of their $n$-categories of operators.
  \item \label{ansatz.dualizable} The algebra of operators in an $(n{+}1)$-dimensional topological order is (at least) $2$-dualizable in $\cat{Mor}_1(n\cat{KarCat}_\bC)$.
  \item \label{ansatz.remote} The algebra of operators in a topological order has trivial centre.
\end{enumerate}
Together, these ans\"atze compile to Definition~\ref{defn.main}. 
We now discuss each ansatz in detail.

It is a fundamental principle of physics, encoded for example in the ``Heisenberg picture'' of quantum mechanics, that a physical system is fully determined, up to equivalence, by its algebra of possible operators/observations --- all physical information should be reconstructible from the algebras, and information about, say, the overall phase of a vector in a Hilbert space, or  the overall normalization of a path integral, is deemed ``unphysical.'' Ansatz~\ref{ansatz.heisenberg} is almost simply an application of this principle to phases of condensed matter. At issue is the restriction just to low-energy topological operators. If we were to include microscopic operators within our algebra, then our algebra would encode not just the low-energy phase of matter, but also its high-energy realization. On the other hand, it is not obvious whether a phase of matter is fully determined by just its low-energy observables. We will assume that it is, but this is an assumption, not a theorem.

Ansatz~\ref{ansatz.category}, declaring that the operators in an $(n{+}1)$-dimensional topological order form a monoidal $n$-category, is largely a declaration of what we mean by ``topological'' in words like ``topological order'': topological operators, by their very nature, may be moved continuously without changing their behaviour, and so may be fused, leading to the categorical structure. There is one subtlety in this logic: it presumes a locality that is not completely obvious in, say, lattice Hamiltonian constructions. Specifically, it presumes that inserting an operator in one region of spacetime does not change the choices of possible operator insertions in some distant region. This locality is surely part of what we mean by ``operator,'' but it might be in conflict with, say, Ansatz~\ref{ansatz.heisenberg}. 
The second part of Ansatz~\ref{ansatz.category}, asserting that the $n$-category $\cA$ of operators is (additive and $\bC$-linear and) Karoubi complete, follows from a version of the construction outlined in \S2.4 of \cite{GJFcond}, which explains how to use a mesoscopic lattice  to build categorical condensations in $\cA$.

Ansatz~\ref{ansatz.stacking} is implicit in Definition~\ref{defn.main}: there are natural symmetric monoidal structures on both $\{${multifusion categories with trivial centre}$\}$ and $\{${topological orders}$\}$, given by Karoubi-completed tensor product and stacking, respectively, and it is reasonable to assume they match. But this assumption, too, could be challenged: perhaps there are operators available in the microscopic models of the two systems which cannot be seen in the low-energy topological limit, but nevertheless combine after stacking into an operator which does have a topological manifestation.

Ansatz~\ref{ansatz.stacking} is needed to justify the dualizability asserted in Ansatz~\ref{ansatz.dualizable}. Take a topological order, described by a monoidal $n$-category $\cA$, and consider ``folding'' it back onto itself.
The resulting ensemble contains two sheets, one described by $\cA$ and the other by the opposite monoidal $n$-category $\cA^\op$, and according to Ansatz~\ref{ansatz.stacking}, this layering is exactly $\cA^e = \cA \boxtimes \cA^\op$. Furthermore, this two-sheet system has a ``fold'' boundary condition which is automatically gapped topological. It determines an $\cA^e$-module (equivalently, a $\cA^e$--$\mathbf 1$ bimodule) via Remark~\ref{remark.stackingdescription}, which is nothing but the canonical module~$\cA$.

Consider now folding our topological order on both the left and the right, so as to produce a zig-zag: 
$$
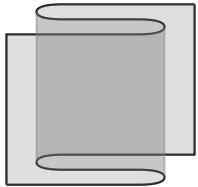
\begin{tikzpicture}
  \draw[thick] 
   (0,0)
    -- ++(1.75,0) 
    .. controls +(.35,0) and +(0,-.02) ..
    ++(.35,.1) 
    .. controls +(0,.02) and +(.35,0) ..
    ++(-.35,.1)
    -- ++(-1,0)
    .. controls +(-.35,0) and +(0,-.02) .. 
    ++(-.35,.1) 
    .. controls +(0,.02) and +(-.35,0) ..
    ++(.35,.1)
    -- ++(1.75,0)
    -- ++(0,2)
  ;
  \draw[thick] 
   (0,0)
    -- ++(0,2)
    -- ++(1.75,0) 
    .. controls +(.35,0) and +(0,-.02) ..
    ++(.35,.1) 
    .. controls +(0,.02) and +(.35,0) ..
    ++(-.35,.1)
    -- ++(-1,0)
    .. controls +(-.35,0) and +(0,-.02) .. 
    ++(-.35,.1) 
    .. controls +(0,.02) and +(-.35,0) ..
    ++(.35,.1)
    -- ++(1.75,0)
  ;
  \draw[fill, opacity=.25, gray, thick]
  (0,0)
    -- ++(1.75,0) 
    .. controls +(.35,0) and +(0,-.02) ..
    ++(.35,.1) coordinate (a)
  -- ++(0,2)
    .. controls +(0,-.02)  and +(.35,0) .. 
    ++(-.35,-.1)
  -- ++(-1.75,0)
  -- ++(0,-2);
  \draw[fill, opacity=.25, gray, thick]
  (a)
    .. controls +(0,.02) and +(.35,0) ..
    ++(-.35,.1)
    -- ++(-1,0)
    .. controls +(-.35,0) and +(0,-.02) .. 
    ++(-.35,.1) coordinate (b)
  -- ++(0,2)
    .. controls +(0,-.02) and +(-.35,0) ..
    ++(.35,-.1)
    -- ++(1,0)
    .. controls +(.35,0) and +(0,.02) ..
    ++(.35,-.1)
  -- (a);
  \draw[fill, opacity=.25, gray, thick]
  (b)
  .. controls +(0,.02) and +(-.35,0) ..
    ++(.35,.1)
    -- ++(1.75,0)
  -- ++(0,2)
  -- ++(-1.75,0)
  .. controls +(-.35,0) and +(0,.02)..
  ++(-.35,-.1)
  -- (b);
\end{tikzpicture}
$$
This zig-zag may be continuously deformed to a flat, unfolded configuration, without a kink. This is the physical reason for the $1$-dualizability of $\cA$, which is automatic for all objects in $\cat{Mor}_1(n\cat{KarCat}_\bC)$. 
What's special about topological orders is that the ``fold'' boundary conditions are topological in the remaining $n$ dimensions, and so may themselves be folded.
The $2$-dualizability of $\cA$ asserted in Ansatz~\ref{ansatz.dualizable} follows from the ability to arrange, and then continuously unkink, configurations like:
$$ 
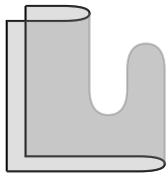
\begin{tikzpicture}
  \draw[thick] 
   (0,0)
    -- ++(1.75,0) 
    .. controls +(.35,0) and +(0,-.02) ..
    ++(.35,.1) 
    .. controls +(0,.02) and +(.35,0) ..
    ++(-.35,.1)
    -- ++(-1.5,0) -- ++(0,2)
  ;
  \draw[thick] 
   (0,0) -- ++(0,2)
    -- ++(.75,0) 
    .. controls +(.35,0) and +(0,-.02) ..
    ++(.35,.1) 
    .. controls +(0,.02) and +(.35,0) ..
    ++(-.35,.1)
    -- ++(-.5,0)
  ;
  \draw[fill, opacity=.25, gray, thick]
  (0,0)
    -- ++(1.75,0) 
    .. controls +(.35,0) and +(0,-.02) ..
    ++(.35,.1)
    -- ++(0,1.25)
    .. controls +(0,.25) and +(.15,0) ..
    ++(-.25,.35)
    .. controls +(-.15,0) and +(0,.25) ..
    ++(-.25,-.35)
    -- ++(0,-.25)
    .. controls +(0,-.25) and +(.15,0) ..
    ++(-.25,-.35)
    .. controls +(-.15,0) and +(0,-.25) ..
    ++(-.25,.35)
    -- ++(0,1)
    .. controls +(0,-.02) and +(.35,0) ..
    ++(-.35,-.1)
    -- ++(-.75,0)
  ;
  \draw[fill, opacity=.25, gray, thick]
  (.25,.2)
    -- ++(1.5,0) 
    .. controls +(.35,0) and +(0,.02) ..
    ++(.35,-.1)
    -- ++(0,1.25)
    .. controls +(0,.25) and +(.15,0) ..
    ++(-.25,.35)
    .. controls +(-.15,0) and +(0,.25) ..
    ++(-.25,-.35)
    -- ++(0,-.25)
    .. controls +(0,-.25) and +(.15,0) ..
    ++(-.25,-.35)
    .. controls +(-.15,0) and +(0,-.25) ..
    ++(-.25,.35)
    -- ++(0,1)
    .. controls +(0,.02) and +(.35,0) ..
    ++(-.35,.1)
    -- ++(-.5,0)
  ;
\end{tikzpicture} $$
In a microscopic lattice description of such a configuration, the microscopic lattice itself may need to twist around.

The final ansatz requiring justification is Ansatz~\ref{ansatz.remote}, which is a formulation of the principle  of
\define{remote detectability} emphasized in \cite{1405.5858,KWZ1,KWZ2}. (This axiomatization of remote detectability builds on ideas from \cite{MR2942952,PhysRevX.3.021009}, and a similar axiom also underlies the analysis in \cite{MR3063919}.) Suppose that $\cA$ is the monoidal $n$-category of extended operators in an $(n{+}1)$-dimensional topological order, and consider the configuration in which some nontrivial operator $X \in \cA$ is inserted. How might this insertion be detected? It has, of course, some microscopic effect, but if it is to be detectable in the low-energy topological limit, then there must be some other operator $Y \in \cA$ implementing the detection. In terms of the monoidal category $\cA$, to say that ``$Y$ detects $X$'' is to say that $Y$ and $X$ have nontrivial commutator, in the higher-categorical sense axiomatized by the Drinfeld centre $\cZ(\cA)$. Remote detectability is the assertion that the only ``invisible'' operators are  multiples of the identity operator. In the Karoubi complete world, ``multiples'' include all objects that can be formed by categorical condensation, hence the condition that $\cZ(\cA) = \Sigma^n\bC$ (and not, say, $\{\id\}$). But, as with the other ans\"atze, Ansatz~\ref{ansatz.remote} requires the physical assumption that all low-energy behaviour of a topological order is detected/determined by low-energy topological operators.

\subsection{Completeness}\label{subsec.complete}

To argue that Definition~\ref{defn.main} is complete, we must construct, for each multifusion $n$-category $\cA$ with trivial centre, an $(n{+}1)$-dimensional ``physical topological order'' whose $n$-category of operators is equivalent to $\cA$, and furthermore we must show that ``physical topological orders'' with equivalent multifusion $n$-categories are ``physically equivalent,'' as required by Ansatz~\ref{ansatz.heisenberg} above.

Such an ``equivalence'' cannot be the equivalence contemplated in the study of ``gapped phases of matter.'' 
Heuristically, a \define{gapped matter system} is a system of quantum matter with finitely many ground states and an energy gap separating those ground states from the next excitation, and a \define{gapped phase of matter} is an equivalence class of gapped matter systems for continuous deformations that do not close or open the energy gap.
 The \define{trivial phase of matter} is the equivalence class that contains systems whose unique ground state factors as a product state (the product being taken over all locations of ``excitation sites'').
 This is only a heuristic definition, since we do not have a complete definition of ``system of quantum matter,'' and since there are some difficulties defining the word ``gapped'' in a fully local way;
 in particular, there are systems which are ``gapped'' for certain definitions, but have far from topological behaviour. We will simply write ``topological (phase of) matter'' for a gapped phase with topological low-energy behaviour.
 
 Although the definitions are only heuristic, much is known about topological phases of matter, and a few features can be extracted immediately. For instance, the definition presupposes that quantum matter systems have well-defined Hilbert spaces, whose ground states may or may not factor as a product state. This allows for the fascinating subject of invertible phases of matter. By definition, a topological quantum matter system $\cX$ is \define{invertible} if there is another system $\cY$ such that the stacking of $\cX$ with $\cY$ is in the trivial phase. Nontrivial invertible phases are distinguished from the trivial phase by the failure of their ground states to factor: the ground states are inherently entangled. 
 
 A feature of invertible phases is that they have no nontrivial topological  operators, of any dimension. 
 To see this, consider stacking an arbitrary $(n{+}1)$-dimensional topological quantum matter system $\cX$ with its orientation-reversal $\cX^\op$. A categorified version of the state-operator correspondence identifies the $n$-category of topological boundary conditions for the stacked system $\cX \otimes \cX^\op$ with the $n$-category of topological operators in $\cX$. But if $\cX$ is invertible, then $\cX \otimes \cX^\op$ is in the trivial phase, and so its $n$-category of topological boundary conditions is ``trivial.''
 The converse, asserting that no topological operators implies invertibility, is almost trivial for functorial TQFTs valued in $n\cat{KarCat}_\bC$. Indeed, such a TQFT is determined by its value $\cX$ on a point, and the multifusion $n$-category of operators is $\End(\cX) = \cX \boxtimes \cX^\op$. If this is trivial, then $\cX$ must have been invertible.
   
  However, the lack of a complete definition of ``topological phase of matter'' makes it hard to prove this converse rigorously in the condensed matter case. One reason to expect it is to run in reverse the argument about boundary operators in the previous paragraph. Indeed, it is generally believed that a topological phase of matter should be determined by its boundary conditions, and perhaps even by (enough information about) a single boundary condition. (See \cite{KWZ1,KWZ2} for detailed discussion and justification of this belief.) In particular, \S2.4 of \cite{GJFcond} constructs a phase $\cY$ from its boundary condition if the boundary condition extends to a categorical condensation $\{\text{trivial phase}\} \condense \cY$. When $\cY = \cX \otimes \cX^\op$, then each step of constructing such an extension involves choosing operators in $\cX$ (again by a version of state-operator correspondence), and the extension might be obstructed if there are nontrivial operators. But when $\cX$ has trivial operators, the extension exists and is essentially unique, and should provide an equivalence between $\cY$ and the trivial phase.
 
 In summary, nontrivial invertible phases are the kernel of the map $\{\text{topological phases}\} \to \{\text{algebras of operators}\}$.
 Because of this, the best that one could hope for is that topological orders, as axiomatized in Definition~\ref{defn.main}, classify a quotient:
$$ \{\text{topological orders}\} \overset?= \frac{\{\text{topological phases}\}}{\{\text{invertible phases}\}}$$
This is the ``completeness'' that we will argue for.

This quotient must be interpreted correctly. One should not merely take the commutative monoid (under stacking) of $(n{+}1)$-dimensional topological phases and quotient by its subgroup of invertible phases, because that does not take into account the higher categorical/homotopical structure of topological and invertible phases. To get the correct homotopical quotient, one must take this naive quotient and add further phases to it. Almost tautologically, $(n{+}1)$-dimensional topological phases can be identified with topological boundary conditions for the trivial $(n{+}2)$-dimensional phase. In our quotient, invertible $(n{+}2)$-dimensional phases are considered trivial. As such, their topological boundary conditions should be added to the set of $(n{+}1)$-dimensional systems under consideration.

One can make this quotient mathematically precise in the world of functorial field theory. An $(n{+}1)$-dimensional \define{framed fully extended TQFT}, with values in a symmetric monoidal $(n{+}1)$-category $\cS$, is a symmetric monoidal functor $Q : \cat{Bord}_{n+1}^{\mathrm{fr}} \to \cS$, where $\cat{Bord}_{n+1}^{\mathrm{fr}}$ is the $(n{+}1)$-category of framed cobordisms constructed in \cite{Lur09,ScheimbauerThesis}. To make contact with quantum mechanics, we require that $\Omega^n\cS = \Vect_\bC$, so that ``$\cS$ looks like $\Vect_\bC$ at the top.'' Now suppose that $\widehat\cS$ is a symmetric monoidal $(n{+}2)$-category with $\Omega\widehat\cS = \cS$. For instance, if $\cS$ is reasonable, then we could take $\widehat\cS = \cat{Mor}_1(\cS)$. Let $\mathbf 1 : \cat{Bord}_{n+1}^{\mathrm{fr}} \to \widehat\cS$ denote the constant functor with value the unit $\mathbf 1 \in \widehat\cS$, and write $\widehat \cS ^\times \subset \widehat\cS$ for the invertible subcategory of $\widehat \cS$.
 Then an \define{anomalous (framed, fully extended) TQFT} consists of a functor $\alpha : \cat{Bord}_{n+1}^{\mathrm{fr}} \to \widehat \cS^\times$ together with a symmetric monoidal lax natural transformation $Q : \mathbf 1 \Rightarrow \alpha$ of functors $\cat{Bord}_{n+1}^{\mathrm{fr}} \to \widehat\cS$. To emphasize: $\alpha$ is required to take invertible values, but $Q$ is not. The $n$-categorical definition of ``lax natural transformation'' is worked out in \cite{JFS}. The term ``anomalous TQFT'' is due to \cite{MR3165462}. The same notion, under the name ``twisted TQFT,'' is explored in \cite{MR2742432}.
 The description of anomalies in terms of invertible field theories is not special to the topological case; see for example Chapter 11 of \cite{MR3969923}.
 
To distinguish from the anomalous case, we will call nonanomalous TQFTs \define{absolute}. By definition, absolute TQFTs have well-defined $\bC$-valued partition functions, and well-defined $\Vect_\bC$-valued Hilbert spaces. Anomalous TQFTs do not. For example, if $(Q,\alpha)$ is an $(n{+}1)$-dimensional anomalous TQFT, then $\alpha(M)$, being invertible, is a one-dimensional vector space for each closed $(n{+}1)$-dimensional spacetime $M$, and so $\alpha$ determines a line bundle on the moduli space of closed spacetimes. The ``partition function'' $M \mapsto Q(M)$ is not $\bC$-valued, but rather it is valued in this line bundle. Even if this line bundle is trivializable, symmetries of $(Q,\alpha)$ might not preserve the trivialization. This is the origin of 't Hooft anomalies. Nontrivializability of $\alpha$ is also possible, and indicates a nonzero gravitational anomaly.

Theorem~7.4 of \cite{JFS} says that an anomalous TQFT $(Q,\alpha)$ together with a choice of trivialization $\alpha \cong \mathbf{1}$ is the same as an absolute TQFT. What is the space of $(Q,\alpha)$ such that $\alpha$ is abstractly trivializable, but a trivialization has not been chosen? It is surjected on by the space of absolute TQFTs: $Q \mapsto (Q,\mathbf{1})$. But suppose that $I$ is an invertible absolute TQFT. It can be used to change the trivialization of $\alpha$, and in that way provides an isomorphism $(Q,\mathbf{1}) \cong (Q\otimes I, \mathbf{1})$ of anomalous TQFTs. We find a bijection of sets:
\begin{align*}
  \{&\text{anomalous $(n{+}1)$D TQFTs} \\ & \text{with trivializable anomaly}\} 
 \\ & \qquad\qquad\qquad =
  \frac {\{\text{absolute $(n{+}1)$D TQFTs}\}} {\{\text{invertible $(n{+}1)$D TQFTs}\}}
\end{align*}
If we quotiented instead by the spectrum of all invertible TQFTs, of all dimensions, we would arrive at anomalous TQFTs with possibly nontrivializable anomaly.

\begin{remark}\label{remark.nonzero}
  For any choice of $\alpha$, there is a natural transformation $Q : \mathbf{1} \to \alpha$ which takes the value $0$ on every nonempty cobordism. A physicist presented with this ``zero TQFT'' will not be able to determine its anomaly, since zero does not change when multiplied by a phase. In particular, a physicist might reasonably declare that ``the anomaly of the zero TQFT is not well-defined.'' We will adopt the convention that \define{anomalous TQFTs} are always nonzero.
\end{remark}

The following result is essentially a corollary of Theorem~\ref{thm.invertibility}, together with some further analysis of $\cat{Mor}_1(n\cat{KarCat}_\bC)$.

\begin{theorem}\label{thm.anomalousTQFT}
  Multifusion $n$-categories with trivial centre, i.e.\ $(n{+}1)$-dimensional topological orders, are equivalent to framed fully-extended $(n{+}1)$-dimensional anomalous TQFTs valued in $\widehat\cS = \cat{Mor}_1(n\cat{KarCat}_\bC)$.
\end{theorem}

\begin{proof}
  The famous Cobordism Hypothesis \cite{BaeDol95,Lur09} identifies framed fully-extended $(n{+}1)$-dimensional TQFTs with $(n{+}1)$-dualizable objects. (Dualizability is at first blush a much weaker condition than the existence of a TQFT taking values on all manifolds: the Cobordism Hypothesis is a deep result about the topology of manifolds, and involves a careful analysis of families of Morse functions.)
  As an application, Corollary 7.7 of \cite{JFS} identifies anomalous $(n{+}1)$-dimensional TQFTs  valued in $\widehat\cS$ with pairs $(X, f)$, where $X \in \widehat\cS^\times$ is an invertible object, and $f : \mathbf{1} \to X$ is a (typically non-invertible) morphism in $\widehat\cS$, such that $f$ is ``$(n{+}1)$-times left adjunctible.''
  
  Which side is the left and which the right?
  In the case of $\widehat\cS = \cat{Mor}_1(n\cat{KarCat}_\bC)$, we do not need to distinguish. Indeed, by Theorem~\ref{thm.invertibility}, $X$ is a multifusion $n$-category with trivial centre; as explained in the first proof of Theorem~\ref{thm.dualizability}, the underlying $n$-category of $X$ is fully dualizable in $n\cat{KarCat}_\bC$. Furthermore, $f$ is a bimodule between the unit object $\mathbf1 = \Sigma^n\bC$ and $X$. 
  One of the two choices of left/right is the statement that $f$ is dualizable over $\mathbf 1$, which is to say that the underlying $n$-category of $f$ is dualizable (hence fully dualizable, by Corollary 4.2.4 of \cite{GJFcond}) in $n\cat{KarCat}_\bC$. The other choice is the statement that $f$ is ``adjunctible over $X$.'' But $X$ is invertible with inverse the opposite algebra $X^\op$, and so $f : \mathbf 1 \to X$ is adjunctible over $X$ if and only if the composition
  $$ X^{\op} = \mathbf 1 \otimes X^\op \overset{f \otimes \id_{X^\op}} \longto X \otimes X^{\op} \isom \mathbf 1$$
   is ``adjunctible over $\mathbf 1$.'' Note that the Morita equivalence $X \otimes X^{\op} \isom \mathbf 1$ is implemented by ``$X$ as an $X \otimes X^{\op}$-module.'' Thus this composition is simply $f$ thought of not as a (right, say) $X$-module but as a (left) $X^\op$-module.
  All together, we find that $f$ is adjunctible, on either side, exactly when its underlying $n$-category is dualizable, in which case it is fully adjunctible.

  As for the Theorem,
  if we are given a multifusion $n$-category $\cA$ with trivial centre, then we may take $X = \cA$ as an object of $\cat{Mor}_1(n\cat{KarCat}_\bC)$, and we may take $f = \cA$ considered as an $\cA$-module. This establishes a map
  $$\{\text{topological orders}\} \to \{\text{anomalous TQFTs}\}.$$
  
  For the inverse map, we are given a multifusion category $X$ with trivial centre, and a fully-adjunctible  $X$-module $f$ which is nonzero by Remark~\ref{remark.nonzero}. The problem is that $f$ is not necessarily ``$X$ as an $X$-module.'' 
  
   We claim that $f : \mathbf 1 \to X$ extends to a condensation $\mathbf 1 \condense X$. Since $X$ is invertible, this is equivalent to claiming that $f \otimes \id_{X^\op} : X^\op \to \mathbf 1$ extends to a condensation. But this follows, for example, from Lemma~\ref{lemma.technical} in the next section. (That Lemma  does not require the present Theorem for its proof.) 
   Furthermore, using Proposition~3.1.5 of \cite{GJFcond},   
   the map $g : X \to \mathbf 1$ participating in the condensation $\mathbf 1 \to X$ may be chosen to be  the right adjoint $g = f^R$.
   In the language of \cite{GJFcond}, this implies that $X$ is determined, up to equivalence, by a unital condensation monad on $\mathbf 1 \in \cat{Mor}_1(n\cat{KarCat}_\bC)$. 
  More down to earth, it says that $X$ is Morita equivalent to a separable unital algebra object $E \in \End(\mathbf 1) = n\cat{KarCat}_\bC$,
  which arises as  the image of a condensation thereon: $f \otimes f^R \condense E$.
  
  Upon unpacking further, one finds that $f^R = \hom_X(f,X)$ as an $X$-module, and that the condensation onto $E$ implements the quotient of $f \otimes f^R$ onto
  $$ E \cong f \otimes_{\End_X(f)} f^R.$$
  Furthermore, the Morita equivalence $X \simeq E$ identifies ``$f$ as an $X$-module'' with ``$f$ as an $E$-module.''
  
  Setting $\cA = \End_X(f) \cong f^R \otimes_X f$, we conclude that 
  $f$ is a Morita equivalence between $\cA$ and $X \simeq E = f \otimes_\cA f^R$.
  This Morita equivalence identifies the pair $(X,f)$ with the pair $(\cA,\cA)$, showing that the map $\{\text{topological orders}\} \to \{\text{anomalous TQFTs}\}$ is an equivalence.
\end{proof}

Although important, Theorem~\ref{thm.anomalousTQFT} does not establish the completeness of Definition~\ref{defn.main}, because \emph{a priori} topological orders and TQFTs are equivalent. It does, though, hint at the proof. 

Suppose $\cA$ is an arbitrary multifusion $n$-category. Then it is fully dualizable in $\cat{Mor}_1(n\cat{KarCat}_\bC)$, and so lives in the subcategory $\Sigma^{n+2}\bC$ by Remark~\ref{remark.bestiary}. 
But then Theorem~2.4.4 of \cite{GJFcond} constructs from $\cA$ an $(n{+}2)$-dimensional gapped topological lattice system $\cX$ whose Hamiltonian is a sum of commuting projectors. Furthermore, the same theorem converts morphisms in $\Sigma^{n+2}\bC$ into interfaces, again of commuting Hamiltonian projector type. After running the argument from \S2.4 of \cite{MR3978827}, we find that if $\cA$ is invertible in $\cat{Mor}_1(n\cat{KarCat}_\bC)$, then $\cX$ is in an invertible phase. We also find that the $1$-morphism $\mathbf{1} \to \cA$ in Theorem~\ref{thm.anomalousTQFT}, namely ``$\cA$ as an $\cA$-module,'' determines a (noninvertible) boundary condition $\cB$ for $\cX$, and the state-operator correspondence identifies $\cA$ with the $n$-category of boundary operators.
This establishes part of the completeness of Definition~\ref{defn.main}: we have built a map
\begin{multline*}
 \{\text{multifusion $n$-categories with trivial centre}\} \\ \longto \frac{\{\text{topological phases}\}}{\{\text{invertible phases}\}}\end{multline*}
whose composition with the function ``consider the $n$-category of operators'' is the identity on $\{\text{multifusion $n$-categories with trivial centre}\}$. In particular, every multifusion $n$-category with trivial centre does arise as the operators on an ``anomalous topological phase.''

To finish the completeness, we must argue that Ansatz~\ref{ansatz.heisenberg}, which claims that physical topological orders are determined by their operators, holds for the quotient $\{\text{topological phases}\} / \{\text{invertible phases}\}$. 
Suppose that $\cB'$ is some $(n{+}1)$-dimensional boundary condition of some $(n{+}2)$-dimensional invertible phase $\cX'$, and that $\cA$ is its multifusion $n$-category of operators. Let $\cB$ and $\cX$ be the phases constructed from $\cA$ as in the previous paragraph. We must show that $\cB$ and $\cB'$ are equivalent ``up to invertible phases.'' To do so, consider the opposite algebra $\cA^\op$, and construct from it the opposite phase $\cB^\op$, which is a boundary condition for $\cX^\op$. Then stack the system $(\cB',\cX')$ with $(\cB^\op,\cX^\op)$ to produce $(\cB' \otimes \cB^\op, \cX' \otimes \cX^\op)$. Assuming Ansatz~\ref{ansatz.stacking} from the previous section, this stacked system will have
$\cA^e = \cA \boxtimes \cA^\op$
 as its $n$-category of (boundary) operators.
 
But $\cA^e$ contains within it a canonical condensation algebra` corresponding to the canonical module $\cA$. Interpreted as operators in $\cB' \otimes \cB^\op$, condensing this algebra produces an interface between $\cB' \otimes \cB^\op$ and a phase with trivial operators. We earlier argued that phases with trivial operators are invertible, and so we have produced a ``boundary condition modulo invertible phases'' for $\cB' \otimes \cB^\op$. This is equivalent to an ``interface modulo invertible phases'' between $\cB'$ and $\cB$. From the construction, this interface is invisible to all operators. Indeed, one can run the construction with the roles of $\cB$ and $\cB'$ exchanged to produce an interface in the other direction, thereby producing the inverse interface. This shows that $\cB$ and $\cB'$ are equivalent in the quotient $\{\text{topological phases}\}/\{\text{invertible phases}\}$, and establishes the completeness of Definition~\ref{defn.main}.

\section{Nondegenerate ground states and braided fusion $(n{-}1)$-categories} \label{sec.braided}

\subsection{From fusion to braided}\label{subsec.fusion}

Our Definition~\ref{defn.main} departs in a small but important way from the definition of ``topological order'' contemplated in \cite{1405.5858,10.1093/nsr/nwv077,PhysRevX.8.021074,PhysRevX.9.021005,PhysRevB.100.045105,1912.13492}, tracking closer to in this way to the definition proposed in \cite{KWZ1,KWZ2}. Algebraically, the distinction stems from the prefix ``multi,'' where by definition a multifusion $n$-category $\cA$ is \define{fusion} if it has no nontrivial $0$-dimensional operators in the sense that $\Omega^n\cA = \bC$ (compare \cite{Reutter2018}, which axiomatizes ``fusion $2$-category'').
Physically, the distinction has to do with whether to allow a local ground state degeneracy: the local ground states are in bijection with the irreducible projections in $\Omega^n\cA$. 

There are a few reasons why we do not impose nondegeneracy of the local ground state as part of our definition of ``topological order.'' Most fundamentally, there are interesting topological orders with local ground state degeneracy; for instance, a local ground state degeneracy can be protected by symmetry. 
Typically, if the symmetry is ignored, then this local ground state degeneracy is unstable, and a flow can be triggered which localizes the topological order onto the suborder supported near a single local ground state. One might think that, in this case, the original topological order would split as some sort of ``direct sum'' of sub-orders supported at each ground state. But this does not hold: except in special circumstances, extra data relating the different ground states is required. In particular, although an individual topological order does not change when it is stacked with an invertible phase, there can be a meaningful ``relative phase'' between two different ground states.

That being said, the case of a nondegenerate local ground state is certainly of interest, and is the focus of this section. Our goal is to prove the intuition, which explains in particular the title of \cite{1405.5858}, that $(n{+}1)$-dimensional topological orders with nondegenerate local ground state are described by braided fusion $(n{-}1)$-categories with trivial braided centre. The argument behind this intuition is that remote detectability, together with the ``fusion'' condition, forbid the presence of nontrivial codimension-$1$ operators, so that our monoidal category $\cA$ ``is'' the braided monoidal category $\Omega\cA$ of operators of codimension $\geq 2$. Indeed, the argument goes, a codimension-$1$ operator must be detectable, and the only way to detect it is by its commutator with a $0$-dimensional operator. More generally, the above cited papers give the impression that there is some sort of perfect pairing between $k$-dimensional and $(n{-}k)$-dimensional operators in an $(n{+}1)$-dimensional topological order.

Although the intuition is correct, the intuitive argument supporting it is problematic. Why must an object of our multifusion $n$-category $\cA$ be detected by an $n$-morphism (a $0$-dimensional operator)? Why can't it be detected by lower-dimensional operators? Indeed, an $(n{+}1)$-dimensional topological order always has nontrivial codimension-$1$ operators. What we will show in the fusion case is not that codimension-$1$ operators are trivial, but that they are built, via categorical condensation, from networks of codimension-$(\geq 2)$ operators. We will not prove the stronger intuition about a perfect pairing between $k$- and $(n{-}k)$-dimensional operators, and the author suspects that it will be hard to make such an intuition precise and that, once made precise, it will be false.

\begin{proposition}\label{prop.braidedstronger}
  Suppose $\cA$ is a fusion $n$-category. Then for each object $A\in\cA$, there is an object $Z \in \cZ(\cA)$ and a condensation $f(Z) \condense A$, where $f : \cZ(\cA) \to \cA$ is the canonical map. 
  \end{proposition}
  
  When $n=1$, this is the statement that the functor $f : \cZ(\cA) \to \cA$ is \define{dominant}.
 
\begin{proof}
 Recall the enveloping algebra $\cA^e = \cA \boxtimes \cA^\op$. We will apply Lemma~\ref{lemma.technical} to the $(n{+}1)$-category $\cS = \Mod(\cA^e)$, pointed by the rank-one free module $\mathbf 1 = \cA^e$. Then $\Omega\cS = \cA^e$, which is fusion since $\cA$ is, so $\Omega^{n{+}1}\cS = \bC$. Thus Lemma~\ref{lemma.technical} assures that any fully-adjunctible nonzero map of $\cA^e$-modules $\cX \to \cA^e$ extends to an $\cA^e$-linear condensation $\cX \condense \cA^e$.
 
 For example, consider the multiplication map $m : \cA \boxtimes \cA \to \cA$ as a map of (left, say) $\cA^e$-modules $\cA^e \to \cA$, and form its $\cA^e$-linear dual map $m^* : \cA^* \to \cA^e$, where $\cA^* = \hom_{\cA^e}(\cA,\cA^e)$. Then $m^*$ extends to a condensation $\cA^* \condense \cA^e$ of (right) $\cA^e$-modules. Tensoring over $\cA^e$ with $\cA$ produces a condensation
 $$ \cA^* \boxtimes_{\cA^e} \cA \condense \cA^e \boxtimes_{\cA^e} \cA = \cA.$$
 But the left-hand side is $\cZ(\cA)$ by compare Remark~\ref{remark.centres}, and map $\cZ(\cA) \to \cA$ produced from $m^*$ is the canonical one, called $f$ above.
 
 Thus $f : \cZ(\cA) \to \cA$ extends to a condensation.  Let $g : \cA \to \cZ(\cA)$ denote the splitting of $f$ in this condensation (which in the proof of Lemma~\ref{lemma.technical} is taken to be the right adjoint of $f$). Then $fg \condense \id_\cA$. Applying these functors to an object $A \in \cA$ produces a condensation $f(g(A)) \condense A$. Setting $Z = g(A) \in \cZ(\cA)$ completes the proof.
\end{proof}

\begin{theorem} \label{thm.braided}
  Suppose $\cA$ is an $(n{+}1)$-dimensional topological order in the sense of Definition~\ref{defn.main}, i.e.\ a multifusion $n$-category with trivial centre, for $n \geq 1$. Then $\cA$ is fusion if and only if the canonical map $\Sigma\Omega\cA \to \cA$ is an equivalence.
\end{theorem}

\begin{proof}
The ``if'' direction is easier, and follows the intuition sketched above. Suppose that $\Sigma\Omega\cA \to \cA$. In other words, for every object $A \in \cA$, there is a condensation $1_\cA \condense A$, so $A$ is selected as the image of a condensation monad $P \in \Omega\cA$. Suppose that $\alpha \in \Omega^n\cA$ is a $0$-dimensional operator. Then it commutes automatically with all $k$-morphisms for $k\geq 1$, i.e.\ with all of $\Omega\cA$. In particular, it commutes with the condensation monad $P$, and hence with its image $A$. Thus if $\cZ(\cA)$ is trivial, there can be no nontrivial $0$-dimensional operators.

The harder ``only if'' direction asserts that if $\cA$ is a fusion $n$-category with trivial centre, then $\cA = \Sigma\Omega\cA$. 
 Equivalently, we wish to show that for each object $A \in \cA$, there is a condensation $1_\cA \to A$, where $1_\cA \in \cA$ denotes the unit object. This follows immediately from Proposition~\ref{prop.braidedstronger} and the assumption that $\cZ(\cA)$ is trivial.
\end{proof}

\subsection{Braided centres}\label{subsec.braidedcentre}

Braided monoidal objects in $(n{-}1)\cat{KarCat}_\bC$ naturally form an $(n{+}2)$-category called $\cat{Mor}_2((n{-}1)\cat{KarCat}_\bC)$. The $1$-morphisms are ``monoidal bimodule $(n{-}1)$-categories'' between braided monoidal $(n{-}1)$-categories, and the $2$-morphisms are bimodule $(n{-}1)$-categories. Every object in $\cat{Mor}_2((n{-}1)\cat{KarCat}_\bC)$ is automatically $2$-dualizable, but $3$-dualizability is hard. Indeed, as in the second proof of Theorem~\ref{thm.dualizability}, one may hope to have a functor $\Sigma : \cat{Mor}_2((n{-}1)\cat{KarCat}_\bC) \to \cat{Mor}_1(n\cat{KarCat}_\bC)$, but it is defined only on those $2$-morphisms which are sufficiently adjunctible. The upshot is that if $\cB$ is $3$-dualizable in $\cat{Mor}_2((n{-}1)\cat{KarCat}_\bC)$, then $\Sigma\cB$ is $3$-dualizable in $\cat{Mor}_1(n\cat{KarCat}_\bC)$, from which full dualizability of $\cB$ follows. 
Extending Definition~\ref{defn.multifusion}, it is reasonable to declare that the fully dualizable objects in $\cat{Mor}_2((n{-}1)\cat{KarCat}_\bC)$ are the \define{braided multifusion $(n{-}1)$-categories}, although as in Remark~\ref{remark.multifusion} in positive characteristic the word ``separable'' should be added.

As an example, suppose that $\cA$ is a multifusion $n$-category. Then $\Omega\cA$ is automatically braided monoidal. Furthermore, the canonical map $\Sigma\Omega\cA \to \cA$ is a full rigid monoidal subcategory. By Proposition~\ref{prop.subfusion}, $\Sigma\Omega\cA$ is fully dualizable in $\cat{Mor}_1(n\cat{KarCat}_\bC)$, and so $\Omega\cA$ is fully dualizable in $\cat{Mor}_2((n{-}1)\cat{KarCat}_\bC)$. In other words, $\Omega\cA$ is braided multifusion.
This fact is strengthened in Corollary~3.33 of~\cite{2011.02859} (which appeared after the first version of this paper was circulated), which says (in the $m=2$ case) that a braided  monoidal $n$-category is braided-multifusion  if and only if its underlying monoidal $n$-category is multifusion.

Although every braided multifusion $(n{-}1)$-category $\cB$, being dualizable, determines an $(n{+}2)$-dimensional TQFT, not every $\cB$ presents an $(n{+}1)$-dimensional topological order. The missing ingredient, as in the multifusion $n$-category case, is a ``remote detectability'' axiom: some sort of ``centre'' of $\cB$ must be trivial. 
Let us write $\cZ_{(1)}(\cB)$ for the centre of $\cB$-as-a-nonbraided-object, i.e.\ its Drinfeld centre. Then $\cZ_{(1)}(\cB)$ is never trivial: the braiding on $\cB$ provides a way to realize every object of $\cB$ inside $\cZ_{(1)}(\cB)$. 
The appropriate centre of a braided monoidal object $\cB$ measures the extent to which objects in $\cB$ commute more than is mandated by the braiding. When we want to emphasize that $\cB$ is considered as a braided object, we will write this centre as $\cZ_{(2)}(\cB)$, and call it the \define{braided centre}; it is also commonly referred (when $\cB$ is a 1-category) as the \define{M\"uger centre} of $\cB$. When the braiding on $\cB$ is implicit, we will simply write $\cZ(\cB)$ and call it the \define{centre} of $\cB$.

There are various equivalent ways to define $\cZ_{(2)}(\cB)$. Most fundamentally, one may consider the $n$-category of endomorphisms of the identity $1$-morphisms on $\cB \in \cat{Mor}_2((n{-}1)\cat{KarCat}_\bC)$. This is the $n$-category of \define{braided $\cB$-modules}.
The identity element is ``$\cB$ as a braided $\cB$-module.'' By definition, $\cZ_{(2)}(\cB)$ is its $(n{-}1)$-category of endomorphisms.
The almost-functor $\Sigma : \cat{Mor}_2((n{-}1)\cat{KarCat}_\bC) \to \cat{Mor}_1(n\cat{KarCat}_\bC)$ is enough to provide an equivalence
$$ \{\text{braided $\cB$-modules}\} = \Omega\{\text{$\Sigma\cB$-bimodules}\},$$
and so \cite[Section 4.8]{HA}
$$ \cZ_{(2)}(\cB) = \Omega\cZ_{(1)}(\Sigma\cB).$$

The analogy between braided $\cB$-modules and $\cA$-bimodules, where $\cA$ is merely associative, may be pressed further. Recall that any monoidal $n$-category $\cA$ has an \define{enveloping algebra} $\cA^e = \cA \boxtimes \cA^\op$, such that $\cA$-bimodules are the same as $\cA^e$-modules. Similarly, any braided $(n{-}1)$-category $\cB$ has an \define{enveloping algebra}, which we will call $\cB^e_{(2)}$, such that braided $\cB$-modules are $\cB^e_{(2)}$-modules. $\cB^e_{(2)}$ is also sometimes called the \define{annular category} of $\cB$, and can be defined as
$$ \cB^e_{(2)} = \int_{S^1_b}\cB$$
where $M \mapsto \int_M \cB$ is the at-least-$2$-dimensional TQFT valued in $\cat{Mor}_2((n{-}1)\cat{KarCat}_\bC)$ determined by $\cB$, and $S^1_b = \partial D^2$ means the circle $S^1$ with boundary framing. For comparison, the associative enveloping algebra could have been defined as $\cA^e = \int_{S^0_b}\cA$, where $S^0_b = \partial D^1$ means the $0$-sphere with boundary framing. Then we may define
$$ \cZ_{(2)}(\cB) = \End_{\cB^e_{(2)}}(\cB).$$
See for example \cite{MR3248737,1804.07538,2006.08022}.

If $\cB$ is multifusion, then it is dualizable as a $\cB^e_{(2)}$-module, and the equivalence $\cZ_{(1)}(\cA) = \int_{S^1_b}\cA$ from Remark~\ref{remark.centres}
extends to:
$$ \cZ_{(2)}(\cB) = \int_{S^2_b}\cB.$$
Indeed, let $\cM$ be the dual module to $\cB$, defined so that the functor $\cM \otimes_{\cB^e_{(2)}}(-) : \cat{Mod}(\cB^e_{(2)}) \to (n{-}1)\cat{KarCat}_\bC$ is right-adjoint to tensoring with $\cB$. Then $\int_{S^2_b}\cB := \cM \otimes_{\cB^2_{(2)}} \cB$. 
But the right adjoint to tensoring with $\cB$ is $\hom_{\cB^e_{(2)}}(\cB,-)$, and so
$$ \cM \otimes_{\cB^2_{(2)}} \cB \cong \hom_{\cB^e_{(2)}}(\cB,\cB) = \cZ_{(2)}(\cB).$$

We will say that a braided multifusion $(n{-}1)$-category $\cB$ \define{has trivial centre} if $\cZ_{(2)}(\cB) = \Sigma^{n-1}\bC$. 
A braided $(n{-}1)$-category with trivial centre is often called \define{nondegenerate}.
By the above remarks, if $\Sigma\cB$ has trivial centre, then so does $\cB$.
We will show the converse in Corollary~\ref{cor.braidedcentre}, confirming Conjecture~2.18 of \cite{KWZ1}.
 Note that the claim is trivial if $n=1$, as then $\cZ_{(2)}(\cB) = \cB$. Otherwise, $n \geq 2$ and $\cB$ is at least $4$-dualizable, so that the TQFT $M \mapsto \int_M \cB$ is at least $4$-dimensional.
In unpublished work reported in \cite{FreedAspects}, Freed and Teleman have shown that an \emph{oriented} at-least-$4$-dimensional TQFT, valued in an arbitrary target higher category, is invertible as soon as its value on $S^2$ is, which would imply the desired claim if the requirement of orientability can be dropped. 
A related invertibility result is due to \cite{BJSS2020}, which appeared after this paper was first posted to arXiv.
Our proof in Corollary~\ref{cor.braidedcentre} is different, and uses facts about multifusion categories that do not hold for  TQFTs with other values.

\subsection{Topological orders with no lines}\label{subsec.nolines}

Theorem~\ref{thm.braided} answered the question of when a topological order $\cA$ can be recovered from its operators of codimension $\geq 2$: exactly when there are no nontrivial $0$-dimensional operators. One can ask the same question in higher codimension: when can $\cA$ be recovered from its codimension-$(\geq 3)$ operators? I.e.\ when does $\cA$ have ``no codimension-$2$ operators'' in the sense that they are all built from codimension-$(\geq 3)$ operators by condensation? The papers \cite{PhysRevX.8.021074,PhysRevX.9.021005} assert without proof (and require the result for their analysis) that the answer is: When $\cA$ has no line operators (other than multiples of the trivial one). Sure enough:

\begin{theorem}\label{thm.nolines}
  Suppose $\cA$ is an $(n{+}1)$-dimensional topological order in the sense of Definition~\ref{defn.main}. Then the canonical map $\Sigma^2\Omega^2\cA \to \cA$ is an equivalence if and only if the category $\Omega^{n-1}\cA$ of line operators in $\cA$ is trivial.
\end{theorem}

To prove the theorem, we will use the following stronger result:

\begin{proposition}\label{prop.nolinesstronger}
  Suppose that $\cB$ is a braided fusion $(n{-}1)$-category, with no restrictions on its centre $\cZ_{(2)}(\cB)$, and suppose that $\Omega^{n-2}\cB $ is trivial. Then for every object $B \in \cB$, there is an object $Z \in \cZ_{(2)}(\cB)$ and a condensation $f(Z) \condense B$, where $f : \cZ_{(2)}(\cB) \to \cB$ is the canonical map.
\end{proposition}

\begin{proof}
As in the proof of Proposition~\ref{prop.braidedstronger}, it suffices to show that $f$ extends to a condensation $\cZ_{(2)}(\cB) \condense \cB$, and we will follow that proof to build it.
Specifically, write $\cB^e = \cB^e_{(2)} = \int_{S^1_b}\cB$ for the enveloping multifusion $(n{-}1)$-category of $\cB$, and let $\cS = \Mod(\cB^e)$ denote its $n$-category of modules. There is a canonical $\cB^e$-linear map $m : \cB^e \to \cB$ described topologically by the inclusion $S^1_b = \partial D^2 \subset D^2$. The dualizability of $\cB$ means that there is a $\cB^e$-linear dual module $\cB^* = \hom_{\cB^e}(\cB,\cB^e)$, and hence a dual map $m^* : \cB^* \to \cB^e$. Tensoring over $\cB^e$ with $\cB$ leads to the map $f : \cZ_{(2)}(\cB) \to \cB$. Thus,
  if we can show that $m^*$ extends to a condensation $\cB^* \condense \cB^e$, then also $f$ will extend to a condensation. 
  
  To complete the proof, it thus suffices to show that the multifusion $(n{-}1)$-category $\cB^e$ is fusion, since then Lemma~\ref{lemma.technical} will imply that $m^*$ extends to a condensation. I.e.\ we wish to compute $\Omega^{n-1}\cB^e = \bC$.
  
  We will do this by using a state-operator correspondence to re-express $\Omega^{n-1}\cB^e$ in terms of the TQFT determined by $\cB^e$.  In fact, $\cB^e$ (and indeed any multifusion $(n-1)$-category) defines not just a TQFT, but a \define{relative} TQFT, which is a version of ``anomalous TQFT'' in which the anomaly $\alpha$ does not need to be invertible \cite{MR3165462,JFS}; compare Theorem~\ref{thm.anomalousTQFT} and the discussion leading up to it. Relative TQFTs can be integrated over not just (framed) manifolds but manifolds with boundary on which the ``boundary condition'' is placed. The characterizing feature of the relative TQFT defined by $\cB^e$ is that $\cB^e$ itself arises as the $(n-1)$-category of boundary excitations.
  
  More generally, relative TQFTs can be evaluated on cobordisms with boundary. These boundaries are separate from the boundaries that cobordisms are stitched along. In order to distinguish, we will write $(M,M_\partial)$ for a typical cobordism with boundary, where  $M_\partial$ is the part of $\partial M$ on which the boundary condition is placed; stitching of cobordisms is done along $\partial(M, M_\partial) := \partial M \sminus M_\partial M$. In particular, a cobordism with boundary is closed if $M_\partial = \partial M$: closed cobordisms with boundary are understood as cobordisms from $\emptyset$ to $\emptyset$. The bulk (aka anomaly) TQFT of $\cB^e$ is the functor
  that takes $(M,\emptyset)$ to what we have been calling $\int_M \cB^e$; it is an $(n+1)$-dimensional absolute TQFT, which is trivial if $\cB$ is a topological order. We will also write this as $\int_{(M,\emptyset)} \cB^e$, so that more generally the value of this TQFT on $(M,M_\partial)$ is $\int_{(M,M_\partial)}\cB^e$.
  
  The \define{state-operator correspondence} is the statement that the vector space of $0$-dimensional operators in an $n$-dimensional TQFT arises as the  Hilbert space of an $(n{-}1)$-dimensional sphere with bounding framing, which we will denote $S^{n-1}_b$. For relative TQFTs, one wants $S^{n-1}_b$ to carry the boundary condition; the bulk manifold is the $n$-disk $D^n$. I.e.:
  $$ \Omega^{n-1}\cB^e = \int_{(D^n, S^{n-1})} \cB^e.$$
  (We leave off the subscript $b$ denoting bounding framing, since $S^{n-1}$ is already understood as $\partial D^n$.) 
  
  Indeed, for any fusion $m$-category $\cF$ and any $k < m$, there is a \define{higher state-operator correspondence} which asserts
  $$ \Omega^k \cF = \int_{(D^{k+1}, S^k)} \cF.$$
  Let us unpack this statement in terms of TQFTs. The fusion $m$-category $\cF$ defines an $(m{+}2)$-dimensional TQFT (with operators $\cZ(\cF)$) together with a distinguished boundary condition. A good way to think of this TQFT is as the ``gauge theory'' which results from ``gauging an $\cF$-symmetry,'' and the boundary condition is the ``Dirichlet'' boundary. (See for instance the discussion of Turaev--Viro theories in \cite{1912.02817}. This way of thinking is completely correct when $\cF$ is generated by a subcategory of invertible operators.)
   We are interested in the $(m-k)$-category of ``states'' for this TQFT that can be placed in the interior of a $(k{+}1)$-dimensional disk and which end at the ``Dirichlet'' boundary condition. With no boundary conditions constraints, the states that can fill a disk are precisely $\cF$. More precisely, for a $(k{+}1)$-dimensional disk with no boundary conditions, the TQFT evaluates to a functor $\int_{D^{k+1}} \cF : \int_{S^k_b}\cF \to \emptyset$. This is a generlized ``hom'' functor: when $k=0$, it is the functor $\cF^\op \boxtimes \cF \to \Sigma^m \bC$ which sends $X \boxtimes Y \mapsto \hom(X,Y)$, and in general this functor measures various spaces of $k$-morphisms. The ``Dirichlet'' boundary condition evaluates this functor at the ``vacuum'' object $1 \in \int_{S^k_b}\cF$, and so computes precisely the end-$k$-morphisms of $1_\cF \in \cF$.
      
  Recall that $\cB^e$ itself arose by compactifying on $S^1_b$ the relative TQFT defined by $\cB$:
  $$ \cB^e = \int_{S^1_b} \cB.$$
  We may therefore compute:
  \begin{align*}
     \int_{(D^n, S^{n-1})} \cB^e & = \int_{(D^n, S^{n-1})} \int_{S^1_b} \cB \\
     & = \int_{(D^n \times S^1_b, S^{n-1} \times S^1_b)} \cB \\
     & = \int_{S^1_b} \int_{(D^n, S^{n-1})} \cB.
  \end{align*}
  But now we can appeal to another state-operator correspondence, this time for the braided fusion $(n-1)$-category $\cB$, or equivalently for the fusion $n$-category $\Sigma \cB$:
  $$ \int_{(D^n, S^{n-1})} \cB = \Omega^{n-1}\Sigma\cB = \Omega^{n-2}\cB.$$
  This is the trivial category $\cat{Vec}$ by assumption, and so $\Omega^{n-1}\cB^e = \int_{S^1_b} \cat{Vec} = \bC$, verifying that $\cB^e$ is fusion.
\end{proof}

\begin{proof}[Proof of Theorem~\ref{thm.nolines}]
  The ``if'' direction is easy: if $\Sigma^2\Omega^2\cA = \cA$, then all of $\Omega^{n-1}\cA$ lifts to $\cZ(\cA)$.
  For the ``only if'' direction, 
  since $\Omega^{n-1}\cA$ is trivial, so is $\Omega^n\cA$, and so, by Theorem~\ref{thm.braided}, $\cA = \Sigma\cB$, where $\cB = \Omega\cA$. Thus it suffices to show that if $\cB$ is an arbitrary braided fusion $(n{-}1)$-category with trivial braided centre, and if its $1$-category $\Omega^{n-2}\cB$ of line operators is trivial, then $\cB = \Sigma\Omega\cB$.
  There is nothing to prove when $n\leq 2$. When $n \geq 3$, this is a special case of Proposition~\ref{prop.nolinesstronger}.
\end{proof}

We may now prove Conjecture 2.18 of \cite{KWZ1}:

\begin{corollary}\label{cor.braidedcentre}
  The assignments $\cA \mapsto \Omega\cA$ and $\cB \mapsto \Sigma\cB$ provide an equivalence between the collections of fusion $n$-categories $\cA$ with trivial centre and braided fusion $(n{-}1)$-categories $\cB$ with trivial centre.
\end{corollary}

\begin{proof}
When $n=1$ there is nothing to show, as in that case $\cB$ is simply a commutative algebra and equal to its braided centre $\cZ_{(2)}\cB = \cB$. For the remainder of the proof, we assume $n\geq 2$.

  Since $\Omega\Sigma\cB = \cB$ for any braided monoidal $\cB$, the only thing to show is that if $\cB$ is a braided fusion
  and $\cZ_{(2)}(\cB)$ is trivial, then $\cZ_{(1)}(\Sigma\cB)$ is trivial.   Equivalently (Theorem~\ref{thm.invertibility}), we must show that $\Sigma\cB$ is invertible in $\Mor_1(n\cat{KarCat}_\bC)$.
  This is equivalent, in turn, to showing that $\cX = \Sigma^2\cB$ is invertible in $(n{+}1)\cat{KarCat}_\bC$.
  Let $\cQ$ denote the corresponding $(n{+}2)$-dimensional TQFT. Its multifusion $(n{+}1)$-category of operators is $\End(\cX) = \cX \boxtimes \cX^* = \cQ(S^0_b)$, which is trivial if and only if $\cX$ is invertible.
  
  Moreover, the $n$-category of codimension-$(\geq 2)$ operators in $\cQ$ is
  $$ \Omega\End(\cX) = \cQ(S^1_b) = \int_{S^1_b} \Sigma\cB = \cZ_{(1)}(\Sigma\cB)$$
  and the $(n{-}1)$-category of codimension-$(\geq 3)$ operators in $\cQ$ is
  $$ \Omega^2\End(\cX) = \cQ(S^2_b) = \int_{S^2_b} \cB = \cZ_{(2)}(\cB).$$
  (These statements are categorical versions of the state-operator correspondence.) By assumption, $\cZ_{(2)}(\cB)$ is trivial and $n \geq 2$. It follows that the $1$-category $\Omega^n\End(\cX)$ of line operators in $\cQ$ is trivial.
  
  But then Theorem~\ref{thm.nolines} implies that $\End(\cX) = \Sigma^2\Omega^2\End(\cX)$ is trivial.
\end{proof}

\begin{remark}\label{remark.highergeneralizations}
  Versions of Theorem~\ref{thm.nolines} and Corollary~\ref{cor.braidedcentre} work in all dimensions. Indeed, let us say that a $p$-monoidal $q$-category $\cC$, Karoubi complete and $\bC$-linear, is \define{multifusion} if it is at least $(p{+}1)$-dualizable in the $(p+q+1)$-category $\cat{Mor}_p(q\cat{KarCat}_\bC)$. (For any reasonable $\cS$, every object in $\cat{Mor}_p(\cS)$ is $p$-dualizable.) Then it is  fully dualizable by a version of Theorem~\ref{thm.dualizability}, whose second proof applies with very few changes. We may further define the \define{$p$-monoidal enveloping algebra} of $\cC$ to be the $1$-monoidal $q$-category
  $$\cC^e_{(p)} = \int_{S^{p-1}_b} \cC,$$
  where the $1$-monoidal structure corresponds to stacking spheres in radial shells. We may also define the \define{$p$-monoidal centre} of $\cC$ to be the $q$-category
  $$\cZ_{(p)}(\cC) = \End_{\cC^e_{(p)}}(\cC) = \int_{S^p_b}\cC.$$
    The second equivalence holds provided $\cC$ is multifusion. This centre
   automatically carries a $(p{+}1)$-monoidal structure from a $(p{+}1)$-dimensional ``pair of pants,'' built by removing two small $(p{+}1)$-dimensional disks from a large $(p{+}1)$-dimensional disk.
   
   Then we may repeat the argument from Theorem~\ref{thm.nolines} to show that $\cC^e_{(p)}$ is fusion provided $\cC$ has no nontrivial operators of dimension $\leq p-1$ in the sense that $\Omega^{q-p+1}\cC$ is trivial. In this case, the canonical functor $\cZ_{(p)}(\cC) \to \cC$ extends to a condensation, and so if $\cC$ has trivial centre, then $\cC = \Sigma\Omega\cC$. We may also repeat the argument from Corollary~\ref{cor.braidedcentre} to show that if $\cC$ has trivial centre, then so does $\Sigma\cC$.
   
   In summary, we learn that a topological order is determined by its operators of codimension $>p$ if and only if it has no nontrivial operators of dimension $<p$. 
   
   Please note what this does not say. A topological order always has operators of low codimension. The statement is merely that when there are no operators of low dimension, then all the operators of low codimension are built via categorical condensation from operators of high codimension. One can picture such operators as networks or meshes. We also are not claiming in general any sort of perfect pairing between equivalence classes of operators of complementary dimension. Indeed, even formulating the perfectness of such a pairing seems nontrivial.
\end{remark}

\section{Classification in low spacetime dimension} \label{sec.dimbydim}

In this last section, we will apply our earlier results, and especially Theorems~\ref{thm.braided} and~\ref{thm.nolines}, to analyze topological orders in low spacetime dimension.
In particular will confirm in \S\ref{subsec.3+1d}, with a few small corrections, the main result of \cite{PhysRevX.8.021074,PhysRevX.9.021005} classifying $(3{+}1)$-dimensional topological orders; the corrections involve a version of ``reduced Galois cohomology'' introduced in \S\ref{subsec.1+1d}.
 It will be interesting to study both the bosonic and fermionic cases (as well as cases enhanced by flavour or time-reversal symmetry). The results in the previous sections apply without change to the fermionic case.

Fermionic topological orders may be defined simply by replacing the category $\Vect_\bC$ with the category $\cat{SVec}_\bC$ of super vector spaces. In particular, rather than working with $\bC$-linear $n$-categories, which are simply $n$-categories enriched in $\cat{Vec}_\bC$, fermionic topological orders are defined in terms of $\cat{SVec}_\bC$-enriched $n$-categories, also called \define{$n$-supercategories}. The difference is that, in an $n$-supercategory, the sets of $n$-morphisms are supervector spaces rather than merely vector spaces, and the compatibility rules relating different types of composition take into account the Koszul sign rules. We also adopt the convention that, for a $1$-supercategory $\cC$ to be \define{additive}, every object $X \in \cC$ must have a parity-reversal object ${\scriptstyle \Pi} X$, defined so that for every $Y \in\cC$, $\hom(Y,{\scriptstyle \Pi}X) = \hom(Y,X) \otimes \bC^{0|1}$; this implies that Karoubi complete additive $n$-supercategories are precisely the $\Sigma^{n-1}\cat{SVec}_\bC$-modules within $n\cat{KarCat}_\bC$.

\subsection{$(0{+}1)$-dimensional topological orders}\label{subsec.0+1d}

According to Definition~\ref{defn.main}, a $(0{+}1)$-dimensional topological order is nothing but a \define{central simple algebra}: a finite-dimensional separable algebra $A$ with one-dimensional centre. Bosonically over $\bC$, and indeed over any algebraically closed field, the only such algebras are matrix algebras: $A \cong \Mat_k(\bC)$ for some $k \in \bN$. Such an algebra arises as the operators in a finite-dimensional topological quantum mechanics model with Hilbert space $\bC^k$.

This is not the end of the story, however, because the isomorphism $A \cong \Mat_k(\bC)$ is not canonical. Of course, the basis of the Hilbert space $\bC^k$ is not canonical. To accommodate this, we can say instead that $A \cong \End(V)$, where $V$ is ``the unique'' irreducible $A$-module. But note the scare-quotes around ``the unique.'' A central simple algebra (over any field) has a unique-up-to-isomorphism irreducible module, but the isomorphism itself suffers a phase ambiguity, and in that sense the vector space $V$ is not quite ``unique.''

This slight failure of uniqueness manifests when considering symmetries of $(0{+}1)$-dimensional topological orders. Suppose $A$ is a $(0{+}1)$-dimensional topological order enhanced with a group $G$ of flavour symmetries, and choose an algebra isomorphism $A \cong \End(V)$, where $V$ is a $k$-dimensional vector space. Then $G$ may not act on~$V$. Indeed, the group of (complex linear) automorphisms of $V$ is $\mathrm{GL}(V) \cong \mathrm{GL}_k(\bC)$, but the group of (complex linear) automorphisms of $A$ is $\PGL(V) \cong \PGL_k(\bC) = \mathrm{GL}_k(\bC) / \bC^\times$. Actions of $G$ on~$A$ are then equivalent to \emph{projective} actions of $G$ on~$V$, and any action of $G$ on~$A$ has an \define{'t Hooft anomaly} in the group cohomology $\H^2(G;\bC^\times)$, which measures the obstruction to lifting the map $G \to \PGL_k(\bC)$ to a map $G \to \mathrm{GL}_k(\bC)$.

There is an equivalent way to say this. Consider replacing $\Vect_\bC$ with the category $\cat{Rep}_\bC(G)$ of (finite-dimensional) $G$-modules. Then we may declare that a \define{$G$-symmetry-enhanced $(0{+}1)$-dimensional topological order} is a central simple algebra object in $\cat{Rep}_\bC(G)$. It is not true that every such algebra arises as $\End(V)$ for some $V \in \cat{Rep}_\bC(G)$. Those that do are the absolute topological orders, and the others have anomaly equal to the $(1{+}1)$-dimensional $G$-SPT corresponding to the 't Hooft anomaly described in the previous paragraph.

Essentially the same story applies to time-reversing symmetries. The only modification is that a time-reversing symmetry (in a unitary physical system) is implemented by a $\bC$-\emph{anti}linear automorphism of $A$. As an example, let us analyze the case of a $\bZ_2^T$ symmetry, meaning a time-reversal symmetry which squares to the identity.  

Working directly, we have an algebra $A$ and an antilinear algebra map $\sigma : A \to A$ such that $\sigma^2 = \id_A$. 
Every $a\in A$ can then be written as $\Re(a) + \sqrt{-1}\Im(a)$, where $\Re(a) = \frac{a + \sigma(a)}2$ and $\Im(a) = \frac{a - \sigma(a)}{2\sqrt{-1}}$ are both $\sigma$-fixed. The result is that the $\sigma$-fixed points are a \define{real form} of $A$, i.e.\ an $\bR$-algebra $A_\bR$ together with an isomorphism $A \cong A_\bR \otimes \bC$.
This $A_\bR$ will be automatically central simple over $\bR$. Conversely, the real algebra $A_\bR$ determines $A = A_\bR \oplus \sqrt{-1} A_\bR$ together with its automorphism $\sigma$, i.e.\ it determines $A$ as a topological order with time-reversal symmetry.

What are the real forms of $A$? Choose an isomorphism $A \cong \Mat_k(\bC)$; then $\sigma(a) = f(\bar a)$, where $a \mapsto \bar a$ is the map that conjugates every matrix entry in $a$ (and so depends on the isomorphism $A \cong \Mat_k(\bC)$), and $f$ is a complex-linear automorphism of $\Mat_k(\bC)$. The statement that $\sigma^2 = \id$ translates into the equation
$$ f \circ \bar{f} = \id,$$
where $\bar f$ is the complex conjugate, in the matrix sense, of $f \in \PGL_k(\bC)$.
This equation can be understood as saying that $(0,1) \mapsto (\id, f)$ defines a 1-cocycle for the twisted nonabelian cohomology $\H^1(\bZ_2^T; \PGL_k(\bC))$. The coboundaries in this cohomology theory come from changes of isomorphism $A \cong \Mat_k(\bC)$, and so the cohomology group measures the real forms, up to isomorphism.

The mathematicians' name for this (nonabelian) cohomology group is $\H^1(\bR; \PGL_k)$. In general, given a field $\bk$ and an algebraic group $G$, there is a (nonabelian) cohomology group $\H^1(\bk;G)$ defined to be equal to the twisted cohomology of the absolute Galois group of $\bk$ with coefficients in the group $G(\bk^s)$ of points of $G$ over the separable closure $\bk^s$ of $\bk$. In the present example, $\bk = \bR$ and $\bk^s = \bC$, and $\bZ_2^T$ is identified with the absolute Galois group of $\bR$.
When $G$ is an abelian algebraic group, there are also higher Galois cohomology groups $\H^i(\bk;G)$.

In general, computing nonabelian Galois cohomology is difficult, and the classification of central simple algebras is rich. But over $\bR$ it is not hard: $\H^1(\bR; \PGL_k)$ is trivial if $k$ is odd, and $\bZ_2$ if $k$ is even. The trivial class in $\H^1(\bR;\PGL_k)$ corresponds to the algebra $\Mat_k(\bR)$, and the nontrivial class in $\H^1(\bR;\PGL_{2k})$ corresponds to the algebra $\Mat_k(\bH)$, where $\bH$ is the quaternion algebra.

The algebraic group corresponding to $\bC^\times$ is called $\bG_m$, and the extension $\bG_m \to \mathrm{GL}_k \to \PGL_k$ of algebraic groups provides a Bockstein map $\H^1(-;\PGL_k) \to \H^2(-;\bG_m)$. The Galois cohomology group $\H^2(\bk;\bG_m)$ is called the \define{Brauer group} of the field $\bk$. When $\bk = \bR$, it is a $\bZ_2$. Just like $(0{+}1)$-dimensional topological orders with $G$ flavour symmetry had 't Hooft anomalies in $\H^2(G; \bC^\times)$, topological orders with time reversal symmetry have anomalies living in  $\H^2(\bZ_2^T; \bC^\times) = \H^2(\bR; \bG_m) \cong \bZ_2$. This anomaly is ``gravitational'' in the sense that the nontrivial class corresponds to a nontrivial $(1{+}1)$-dimensional time-reversal phase, called the \define{Haldane chain}. If one were to define $(0{+}1)$-dimensional topological orders over other fields, one would find a ``gravitational anomaly'' living in the Brauer group.

Note that, just like the $G$-symmetric case could have been described by replacing $\Vect_\bC$ with the category $\cat{Rep}_\bC(G)$ of $G$-modules, the time-reversal case could be described by replacing $\Vect_\bC$ with $\Vect_\bR$, since time-reversal $(0{+}1)$-dimensional phases are identified with central simple algebras therein. In fact, these two replacements are in perfect parallel. The category $\cat{Rep}_\bC(G)$ arises as the ``categorical fixed points'' of the trivial $G$-action on $\Vect_\bC$, whereas \define{Galois descent} identifies $\Vect_\bR$ with the category of fixed points of the action of complex conjugation on $\Vect_\bC$. (By definition, a fixed point of an action of a group $G$ on a category $\cC$ is an object $X \in \cC$ together with isomorphisms $\phi(g) : X \isom g(X)$ such that $\phi(gh) = \phi(g)\phi(h)$.)

Fermionic $(0{+}1)$-dimensional topological orders may be analyzed similarly: now Definition~\ref{defn.main} identifies them with central simple \emph{super}algebras. The major difference is that not every such algebra is a matrix algebra. The complete classification is the following: there are the matrix algebras $\Mat_{p|q}(\bC) \cong \Mat_{q|p}(\bC)$, indexed by an (unordered!) pair $(p|q)$, which is the superdimension of an irreducible module; and there are the algebras $\Mat_k(\Cliff(1))$, where $\Cliff(1) = \bC\langle x\rangle / (x^2 = -1)$ is the \define{Clifford algebra} on one fermionic generator $x$ (and there is only one index $k \in \bN$).

These two sets correspond to the two classes in the \define{super Brauer group} of $\bC$, which is the group of gravitational anomalies for $(0{+}1)$-dimensional topological orders. The anomaly truly is gravitational. Fermionic theories (if they are unitary) can only be placed on manifolds with spin structure. But a topological order with operator algebra $\Mat_k(\Cliff(1))$ does not have a well-defined partition function: on circles with odd spin structure, the ``partition function'' is not a number but an element of the odd line $\bC^{0|1}$, and on circles with even spin structure, the ``partition function'' changes sign under the canonical automorphism of the spin structure. The anomaly itself can be understood as the invertible $(1{+}1)$-dimensional \define{Arf} theory, also called the \define{Majorana chain}.

These examples can be mixed. One can, for example, allow fermions and also a time-reversal symmetry, enjoying either $\sigma^2 = \id$ or $\sigma^2 = (-1)^{F}$.
The classification is then given by the appropriate \define{Galois supercohomology} group, a notion which is straightforward to write down but does not seem to have been studied in the literature.

\subsection{$(1{+}1)$-dimensional topological orders}\label{subsec.1+1d}

A $(1{+}1)$-dimensional topological order is a multifusion ($1$-)category with trivial centre. Theorem~\ref{thm.braided} implies that a $(1{+}1)$-dimensional topological with a nondegenerate local ground state is trivial. This suggests that a general $(1{+}1)$-dimensional topological order is determined, as a product state, by its local ground state degeneracy; compare 
\cite{PhysRevB.83.035107,PhysRevB.84.165139,PhysRevB.84.235128,KWZ1,KLWZZ}.
We will see that this is essentially correct, with some important caveats.

Suppose that $\cA$ is a $(1{+}1)$-dimensional topological order. 
The algebra $\Omega\cA$ of $0$-dimensional operators is a commutative separable finite-dimensional algebra over $\bC$. Any such algebra is automatically \emph{canonically} isomorphic to a direct sum of copies of $\bC$, where the sum is indexed by the \define{spectrum} $\Spec(\Omega\cA) = \hom(\Omega\cA,\bC)$, where the hom set is in the category of commutative $\bC$-algebras.

Arbitrarily choose a point $s \in \Spec(\Omega\cA)$. This choice is equivalent to a choice of indecomposable projector $\delta_s \in \Omega\cA$, or equivalent a simple direct summand $\mathbf 1_s = \delta_s \mathbf 1 \delta_s$ of the unit object $\mathbf 1 \in \cA$. This object $\mathbf1_s$ is automatically a separable associative algebra object internal to $\cA$.

We will condense this algebra $\mathbf 1_s$. In the abstract language of \cite{GJFcond}, this means to use $\mathbf 1_s$ to build a condensation monad supported by $\cA \in \Mod(\cA)$, and to take the image of that condensation monad. Concretely, it means to consider the category $\cM_s = \Mod_\cA(\mathbf 1_s)$ of $\mathbf 1_s$-module-objects in $\cA$. This $\cM_s$ is a nonzero dualizable $\cA$-module, and the \define{condensation} of $\mathbf 1_s$ is the monoidal category $\End_\cA(\cM_s)$, and $\cM_s$ is a Morita equivalence $\cA \simeq \End_\cA(\cM_s)$ as shown in the proof of Theorem~\ref{thm.anomalousTQFT}.
Since $\cM_s = \Mod_\cA(\mathbf 1_s)$, the monoidal category $\End_\cA(\cM_s)$ is equivalent to the category $\cat{Bimod}_\cA(\mathbf 1_s)$ of $\mathbf 1_s$-bimodule-objects in $\cA$. Physically, $\End_\cA(\cM_s)$ describes a new $(1{+}1)$-dimensional topological order formed from $\cA$ by condensing the anyon $\mathbf 1_s$, and $\cM$ describes an interface between the new and old topological orders.

What are the $0$-dimensional operators in $\End_\cA(\cM_s)$? The equivalence $\End_\cA(\cM_s) \simeq \cat{Bimod}_\cA(\mathbf 1_s)$ identifies the $0$-dimensional operators in $\End_\cA(\cM_s)$ with the endomorphisms of $\mathbf 1_s$ as a $\mathbf 1_s$-bimodule. But $\mathbf 1_s$ is simple, and so this endomorphism ring is just $\bC$. By Theorem~\ref{thm.braided}, $\End_\cA(\cM)$ must be trivial. But then $\cA \cong \End(\cM_s)$, the endomorphisms of $\cM \in \cat{KarCat}_\bC$.

As a category, how does $\cM_s$ look? A $\mathbf 1_s$-module is a bimodule between $\mathbf 1_s$ and the unit $\mathbf 1$, which decomposes as a sum over $\Spec(\Omega\cA)$. This implies that $\cM_s$ decomposes as a direct sum of categories indexed by $\Spec(\Omega\cA)$:
$$ \cM_s = \bigoplus_{s' \in \Spec(\Omega\cA)} \cM_{s,s'}$$
where
$$ \cM_{s,s'} = \{\text{$\delta_s$-$\delta_{s'}$ bimodule objects in $\cA$}\}. $$
This means that $\End(\cM_s)$ decomposes as a ``matrix category'':
$$ \End(\cM_s) = \bigoplus_{s',s'' \in\Spec(\Omega\cA)} \hom( \cM_{s,s'}, \cM_{s,s''}).$$
On the other hand, without choosing a basepoint $s \in \Spec(\Omega\cA)$, we can canonically decompose
$$ \cA = \bigoplus_{s',s'' \in\Spec(\Omega\cA)}  \cM_{s',s''} $$
These two decompositions are identified in the equivalence $\cA \cong \End(\cM_s)$. Focussing on the endomorphisms of the unit, we find that 
$$ \Omega\cA \cong \Omega\End(\cM_s) = \bigoplus_{s' \in \Spec(\Omega\cA)} \End(\cM_{s,s'}) $$
as $\Spec(\Omega\cA)$-graded algebras. 
This implies that each summand category $\cM_{s,s'}$ is invertible in $\cat{KarCat}_\bC$.

Bosonically, the only invertible object in $\cat{KarCat}_\bC$ is the unit $\Vect_\bC$ itself. 
We learn, therefore, that
\begin{gather*}
\cM_s \cong \bigoplus_{\Spec(\Omega\cA)} \Vect_\bC
\end{gather*}
and $\cA = \End(\cM_s)$ is the corresponding $\Spec(\Omega\cA) \times \Spec(\Omega\cA)$ matrix category. This is the sense in which $\cA$ is determined by its local ground state degeneracy.

Fermionically, i.e.\ with $\cat{KarCat}_\bC$ replaced by the $2$-category $\cat{KarSCat}_\bC$ of Karoubi complete supercategories,
there are two invertible objects: the unit $\cat{SVec}_\bC$ and the supercategory of $\Cliff(1)$-modules. We can indicate which choice is $\cM_{s,s'}$ by equipping $\Spec(\Omega\cA)$ with a $\bZ_2$-valued function. Even better is to recognize the $\bZ_2$ of invertible objects in $\cat{KarSCat}_\bC$ as the bottom ``Majorana'' layer of extended supercohomology $\SH^\bullet$, so that this $\bZ_2$-valued function is really a class $\alpha \in \SH^2(\Spec(\Omega\cA))$. Extended supercohomology is due to \cite{WangGu2017}; see \S\S5.4--5.5 of
\cite{MR3978827} for definitions and the relation to the Morita theory of Clifford algebras.

How canonically can $\cA$ be recovered from the set $\Spec(\Omega \cA)$ and, in the fermionic case, this $\bZ_2$-valued function or supercohomology class? Two issues present themselves. First, in order to construct $\cM_s$ and thereby to identify $\cA$ with its endomorphism category, we needed to choose a ``base point'' $s \in \Spec(\Omega\cA)$. This choice might not be preserved by a flavour or time-reversal symmetry. Indeed, we expect to find some sort of 't Hooft anomalies when studying topological enriched by symmetry, and there would be no room for such anomalies if $\cA$ were fully-canonically determined by $\Spec(\Omega\cA)$.
 Second, in the fermionic case, the supercohomology class vanishes at the basepoint $s$, since $\cM_s = \cat{Bimod}_\cA(\delta_s\mathbf 1) \cong \SVect_\bC$ and not $\Mod(\Cliff(1))$.
 
 Handling the latter issue first, note that, since $\Mod(\Cliff(1))$ is invertible, for any supercategory $\cM$, 
 $$ \End(\cM) = \End\bigl(\Mod(\Cliff(1)) \boxtimes \cM\bigr).$$
This means that the actual function $\alpha : \Spec(\Omega\cA) \to \bZ_2$ is not needed, but only its class modulo changing $\alpha(s') \mapsto \alpha(s')+1$ globally. Anticipating the result of further analysis, the cohomological way to say this is that we require only the class $[\alpha]$ as an element of \define{reduced supercohomology} $\widetilde{\SH}{}^2(\Spec(\Omega\cA)) = \SH^2(\Spec(\Omega\cA)) / \SH^2(\pt)$. 

In order to investigate the former issue, let us analyze bosonic $(1{+}1)$-dimensional topological orders enhanced by either a ($\bC$-linear) $\bZ_2$ flavour symmetry or a ($\bC$-antilinear) $\bZ_2^T$ time-reversal symmetry. Galois descent, mentioned in \S\ref{subsec.0+1d}, identifies time-reversal-symmetric $(1{+}1)$-dimensional topological orders $(\cA, \bZ_2^T\text{ action})$ with  $(1{+}1)$-dimensional topological orders $\cA_\bR$ defined over $\bR$.

If the topological order $\cA$ has a unique local ground state, then Theorem~\ref{thm.braided} continues to apply: there is a canonical equivalence $\cA \cong \Vect_\bC$, preserved by any symmetry, and so the $\bZ_2$ or $\bZ_2^T$ symmetry is trivial. (More precisely, the $\bZ_2^T$ symmetry acts on $\Vect_\bC$ by complex conjugation.) This should be contrasted with the case of absolute $(1{+}1)$-dimensional topological phases. A $(1{+}1)$-dimensional phase with a unique ground state may be  protected from triviality by a $G$-symmetry, i.e.\ it  may by a \define{$G$-SPT}. The classification of $(1{+}1)$-dimensional $G$-SPTs is well-known to be given by the group cohomology $\H^2(G;\bC^\times)$. When $G = \bZ_2$, this group is trivial. But when $G = \bZ_2^T$, i.e.\ when $G$ acts nontrivially on the coefficients $\bC^\times$, then $\H^2(\bZ_2^T; \bC^\times) \cong \bZ_2$ is nontrivial. The two classes are the trivial $(1{+}1)$-dimensional phase and the Haldane chain mentioned earlier.

Suppose next that the set $\Spec(\Omega\cA)$ of local ground states has order $2$. This set is acted on by the $\bZ_2$- or $\bZ_2^T$-symmetry. Consider first the case when the action is trivial.
If we cared about absolute phases, we would find a $G$-SPT at each ground state: in the $\bZ_2$ case there is no data, but in the $\bZ_2^T$ case there would be $2^2$ choices. But for our topological orders, the analysis proceeds just like in the fermionic case studied above: the only information in the topological order is the ``relative phase'' between the ground states, so in the $\bZ_2$ flavour case, there is only the trivial option, and in the $\bZ_2^T$ case, the $2^2$ absolute choices collapse to $2$ possible topological orders. 

Finally, suppose that $\Spec(\Omega\cA)$ is a set of size two, exchanged by the $\bZ_2$- or $\bZ_2^T$-action. It is not hard to show in the case of absolute phases that there are no further choices: there is a unique absolute $(1{+}1)$-dimensional phase with two ground states exchanged by $\bZ_2$- or $\bZ_2^T$-symmetry. For topological orders with $\bZ_2^T$-symmetry, this result remains true. Then $\Spec(\Omega\cA)$ with its time-reversal symmetry corresponds to ``$\bC$ as a $\bR$-algebra,'' and the only $\bR$-linear topological order with this ring of zero-dimensional operators is the category of $\bR$-linear $\bC$-$\bC$ bimodules. But for the $\bZ_2$ flavour symmetry, there are two choices, parameterized by $\H^3(\bZ_2; \bC^\times) \cong \bZ_2$. To construct them, recall that a class $\alpha \in \H^3(\bZ_2; \bC^\times)$ determines a fusion category $\Vect^\alpha_\bC[\bZ_2]$ with $\bZ_2$ fusion rules. (The nontrivial class provides the fusion category for the \define{semion}. The two fusion categories are equal as linear categories.) Let $1$ and $X$ denote the simple objects in $\Vect^\alpha_\bC[\bZ_2]$. Then $Y \mapsto X \otimes Y$ does not define a $\bZ_2$-action on $\Vect^\alpha_\bC[\bZ_2]$ when $\alpha \neq 0$. But it does determine a $\bZ_2$-action $f(-) \mapsto X \otimes f(X \otimes (-))$ on the category of endomofunctors of $\Vect^\alpha_\bC[\bZ_2]$, and the two choices for $\alpha$ provide different actions. This is an analogue of the following fact about twisted group algebras. Given a finite group $G$ and a class $\alpha \in \H^2(G;\bC^\times)$, construct the twisted group algebra $\bC^\alpha[G]$; then $G$ does not act by multiplication on $\bC^\alpha[G]$, but does 
acts on the matrix algebra $\Mat(\bC^\alpha[G])$.

To summarize these examples, and to state the complete classification, it is best to switch to cohomological language.  A (extraordinary) cohomology theory $h^\bullet$ assigns abelian groups $h^\bullet(X)$ to topological spaces $X$. If $G$ is a finite group, then we can consider also an \define{equivariant} cohomology theory $h_G^\bullet$ defined on the category of $G$-spaces, i.e.\ topological spaces with a $G$-action. (Constructing equivariant extensions of a extraordinary cohomology theory $h^\bullet$ is typically a difficult thing, but when the spectrum representing $h^\bullet$ is coconnective, then we may use \define{Borel} equivariant cohomology $h_G^\bullet(X) = h^\bullet(X\sslash G)$, where $X$ is a $G$-space and $X \sslash G = X \times_G EG$ for some contractible space $EG$ with a free $G$-action. One reason Borel-equivariant cohomology is safe  is because the requisite spectral sequences has very good convergence in the coconnective case.)
For example, absolute bosonic $(1{+}1)$-dimensional  phases enriched by $G$-symmetry are classified by equivariant ordinary cohomology $\H^2_G(\Spec(\Omega\cA);\bC^\times)$. Absolute fermionic $(1{+}1)$-dimensional phases are classified by equivariant supercohomology $\SH^2_G(\Spec(\Omega\cA))$.

There is a unique map $\pi : X \to \pt$ for each space $X$, and if $X$ is a $G$-space, then this map is $G$-equivariant. This determines maps
$$ \pi^* : h^\bullet(\pt) \to h^\bullet(X), \qquad \pi^* : h^\bullet_G(\pt) \to h^\bullet_G(X).$$
The first of these is always an injection (if $X$ is nonempty), and \define{reduced $h$-cohomology} is by definition the quotient
$$ \widetilde{h}^\bullet(X) = \frac{h^\bullet(X)}{h^\bullet(\pt)}.$$
Indeed, any choice of  basepoint $x \in X$ determines a splitting $h^\bullet(X) \cong \widetilde{h}^\bullet(X) \oplus h^\bullet(\pt)$.
But for $G$-spaces, $\pi^*$ may not be an injection. (It is an injection of there is a $G$-equivariant map $\pt \to X$, i.e.\ if $X$ contains any $G$-fixed points, as then $h_G^\bullet(X)$ contains $h_G^\bullet(\pt)$ as a direct summand.) Because of this, one should not take the ordinary quotient above in order to define ``reduced equivariant cohomology,'' but rather one must take a ``homotopy quotient'':
\begin{definition}
  Let $h_G^\bullet$ be an equivariant cohomology theory. If $X$ is a $G$-space, then the group $\widetilde h_G^\bullet(X)$ of \define{reduced equivariant cohomology} classes on $X$ consist of  pairs $(\alpha,\beta)$ where $\alpha \in h_G^{\bullet+1}(\pt)$ and $\beta$ is trivialization of $\pi^*\alpha \in h_G^{\bullet+1}(X)$. 
  I.e.\ in a cochain model, 
  $\alpha$ is a degree-$(\bullet+1)$ cocycle, 
  $\beta$ is a 
  degree-$\bullet$ cochain which is not a cocycle but rather solves $\d\beta = \pi^*\alpha$, and $(\alpha,\beta)$ is cohomologous to $(\alpha + \d a, \beta + \pi^* a + \d b)$.
\end{definition}
In the nonequivariant case, this definition recovers the ``quotient'' definition above: $\pi^*$ is an injection, so for any class $(\alpha,\beta) \in \widetilde h^\bullet(X)$, the class $\alpha$ is already trivializable in $h^{\bullet+1}(\pt)$; but the ambiguity in the choice of trivialization $\alpha = \d a$ means that a cocycle $\beta \in h^\bullet(X)$ is considered cohomologous to $\beta + \pi^* a$ for any cocycle $a \in h^\bullet(\pt)$.
But in general, rather than being a mere quotient, reduced equivariant cohomology participates in a long exact sequence, with connecting map $(\alpha,\beta) \mapsto \alpha$:
$$ \cdots \to h^\bullet_G(\pt) \overset{\pi^*}\to h^\bullet_G(X) \to \widetilde h_G^\bullet(X)   \to h^{\bullet+1}_G(\pt) \overset{\pi^*}\to \cdots $$
For example, when $X = G$ carries the free $G$-action (by multiplication), then the equivariant ordinary cohomology groups $\H^2_G(X; \bC^\times)$ and $\H^3_G(X;\bC^\times)$ vanish, regardless of how $G$ acts on the coefficients $\bC^\times$. Therefore the connecting map $\widetilde{\H}{}^2_G(X;\bC^\times) \to \H^3_G(\pt;\bC^\times)$ is an isomorphism. When $G = \bZ_2^T$ acting $\bC$-antilinearly on $\bC^\times$, the group $\H^3_G(\pt;\bC^\times)$ is trivial, but when $G = \bZ_2$, $\H^3_G(\pt;\bC^\times) = \bZ_2$. 

\begin{theorem}\label{thm.1+1d}
  Bosonic $(1{+}1)$-dimensional topological orders $\cA$ with $G$-symmetry are classified by a $G$-action on the set $\Spec(\Omega\cA)$ together with a class in reduced equivariant ordinary cohomology $\widetilde{\H}{}^2_G(\Spec(\Omega\cA); \bC^\times)$. Fermionic $(1{+}1)$-dimensional topological orders $\cA$ with $G$-symmetry are classified by a $G$-action on the set $\Spec(\Omega\cA)$ together with a class in reduced equivariant  supercohomology $\widetilde{\SH}{}^2_G(\Spec(\Omega\cA))$.
\end{theorem}

\begin{proof}
  The equivalence in Theorem~\ref{thm.anomalousTQFT} between topological orders and anomalous TQFTs is universal, and holds in the presence of any symmetry. In particular, a $G$-action on the topological order $\cA$ is the same as a $G$-action on the invertible object in $\cat{Mor}_1(\cat{KarCat}_\bC)$ (or $\cat{Mor}_1(\cat{KarSCat}_\bC)$ in the fermionic case) represented by $\cA$, together with an extension of that action to ``$\cA$ as an $\cA$-module.''
  
  The analysis at the start of this subsection implies that $\cA$ is equivalent in $\cat{Mor}_1(\cat{KarCat}_\bC)$ to the trivial object $\Vect_\bC$ (or $\SVect_\bC$ in the fermionic case). The automorphism 3-group of the trivial object $\Vect_\bC \in \cat{Mor}_1(\cat{KarCat}_\bC)$ is the group $\cat{KarCat}_\bC^\times$ of invertible objects in $\cat{KarCat}_\bC$, which is just $\bC^\times$ in degree $2$; in the fermionic case, it is the classifying space of supercohomology (compare \S5.5 of \cite{MR3978827}). Thus actions of $G$ on this object are classified by $\alpha \in \H^3_G(\pt;\bC^\times)$ in the bosonic case and by $\alpha \in \SH^3_G(\pt)$ in the fermionic case.
    
  The analysis at the start of this subsection furthermore says that, after applying the Morita trivialization of $\cA$, the $\cA$-module $\cA$ becomes equivalent to $\bigoplus_{\Spec(\Omega\cA)} \Vect_\bC$ in the bosonic case. In the fermionic case, it becomes equivalent to a sum indexed by $\Spec(\Omega\cA)$ of copies of $\SVect_\bC$ and of $\Mod(\Cliff(1))$. $G$-actions on this category compatible with the chosen $G$-action on the indexing set $\Omega\cA$ with the chosen $G$-action $\alpha$ on the trivial object are classified by trivializations $\beta$ of $\pi^*\alpha$.
\end{proof}

Note that, in the proof of Theorem~\ref{thm.1+1d}, the connecting map $\widetilde h^2_G(\Spec(\Omega\cA)) \to h^3_G(\pt)$ takes a $G$-symmetry enhanced topological order $\cA$ to the $G$-SPT which describes the ``anomaly'' of the anomalous TQFT. In other words, this connecting map assigns to the topological order its \define{'t Hooft anomaly}.

For any space $X$ with finite homotopy, a cohomology class $\beta \in \H^2(X; \bC^\times)$ (respectively $\SH^2(X)$) is precisely the data needed in order to define a $(1{+}1)$-dimensional bosonic (respectively fermionic) \define{topological sigma model} with target $X$ \cite{PhysRevB.100.045105}, via a Dijkgraaf--Witten type construction. This construction always produces absolute TQFTs. The topological order constructed from a reduced cohomology class can be thought of, by analogy, as an \define{anomalous topological sigma model}. Thus Theorem~\ref{thm.1+1d} can be summarized as: Each $(1{+}1)$-dimensional topological order $\cA$ is, canonically, an anomalous topological sigma model with target $\Spec(\Omega\cA)$. In the presence of $G$ flavour symmetry, $\Spec(\Omega\cA)$ is thought of as a $G$-set, and anomalous topological sigma models are classified by reduced $G$-equivariant cohomology. In the presence of a time-reversing $\bZ_2^T$, Galois descent says that $\Spec(\Omega\cA) \sslash G$ is best thought of not as a $\bZ_2$-set, but rather as the  \define{scheme} $\Spec(\Omega\cA_\bR)$, and anomalous sigma models are classified by \define{reduced Galois cohomology}.

\subsection{$(2{+}1)$-dimensional topological orders} \label{subsec.2+1d}

Corollary~\ref{cor.braidedcentre} implies that $(2{+}1)$-dimensional topological orders with nondegenerate local ground states are classified by braided fusion ($1$-)categories with trivial braided centre. This  recovers the well-known classification (compare \cite{MR1002038,MR1146945,MR1016869,MR1104414,MR1199171,MR2200691,10.1093/nsr/nwv077}) that $(2{+}1)$-dimensional topological orders are classified by unitary modular tensor categories, except for the issue of unitarity. By definition, a \define{modular} tensor category is a braided fusion category with trivial centre which is equipped with a ribbon structure (the axioms of which we will not review). The ribbon structure allows the corresponding (anomalous) $(2{+}1)$-dimensional TQFT to be placed on any oriented manifold: the TQFT is \emph{isotropic}, and does not require a spin structure or framing in order to be defined. This ribbon structure is canonically determined if the category is equipped with a unitary structure (the axioms of which we also will not review). In other words, to the extent that unitarity can be physically justified, in the $(2{+}1)$-dimensional case the corresponding TQFT is automatically isotropic. Note that this is somewhat surprising, as microscopic models of topological orders are usually defined on lattices, and are essentially never isotropic: any isotropy of the topological order only emerges in the low energy. Similarly in higher dimensions, the isotropic should be a theorem and not part of the definition. So far we are lacking a sufficient understanding of ``unitary higher categories'' to produce such a theorem.

When there are multiple local ground states, our analysis follows (the beginning of) \S\ref{subsec.1+1d}. Choose any local ground state $s \in \Spec(\Omega^2\cA)$. This choice corresponds to a simple sub-object of the identity line operator $\mathbf 1 \in \Omega\cA$. This object may be condensed, and the resulting system will have a unique local ground state, and so will be described by a braided fusion category $\cB_s$ with trivial centre. For any pair $s',s'' \in \Spec(\Omega^2\cA)$, the $2$-category of surface operators (and interfaces thereof) from $s'$ to $s''$ determines a Morita equivalence $\Sigma\cB_{s'} \simeq \Sigma\cB_{s''}$ in $\cat{Mor}_1(2\cat{KarCat}_\bC)$, or equivalently a Morita equivalence (up to Morita equivalence thereof) $\cB_{s'} \simeq \cB_{s''}$ in $\cat{Mor}_2(\cat{KarCat}_\bC)$. Thus arbitrary $(2{+}1)$-dimensional topological orders are classifiable provided we can classify braided fusion categories with trivial centre, and their Morita equivalences.

The actual classification of braided fusion categories with trivial centre is almost surely wild, but the classification up to Morita equivalence is reasonably well understood. Following \cite{MR3039775}, we will write $\cW$ for the group of Morita equivalence classes of braided fusion categories with trivial centre, and refer to it as the \define{bosonic Witt group}. If we were to work instead with supercategories, we would have the \define{fermionic Witt group} $s\cW$ studied (under the name ``slightly degenerate Witt group'') in \cite{MR3022755}. The name  comes from the fact that any anisotropic metric abelian group  determines a $(2{+}1)$-dimensional topological phase (indeed, an abelian Chern--Simons theory), and in this way $\cW$ contains as a subgroup the ``classical Witt group'' of aniostropic metric abelian groups (\S5.3 of \cite{MR3039775}). The main result of  \cite{MR3022755} is a rather complete description of $s\cW$, which is slightly easier to analyze than is $\cW$ itself. The map $\cW \to s\cW$ that forgets the ``bosonicness'' of a bosonic topological order is considered in \S5.3 of \cite{MR3022755}, where it is shown to have kernel $\bZ_{16}$. The cokernel is not known.

\begin{remark}\label{remark.fermioniocWittgroup}
Because we will use it in \S\ref{subsec.3+1d}, let us observe that $\{$trivial-centre braided fusion supercategories up to Morita equivalence$\}$ is not best thought of as a set, but rather as $4$-groupoid: the $1$-morphisms are the super Morita equivalences themselves, the $2$-morphisms are the equivalences of equivalences, etc. This $4$-groupoid is symmetric monoidal, and so defines a spectrum. We will write $\underline{s\cW}$ for this spectrum, indexed so that the $4$-morphisms contribute homotopy in degree $0$. The homotopy groups of this spectrum are then concentrated in degrees $\pi_{-4},\dots,\pi_0$, and $\pi_{-n}\underline{s\cW} = \{$equivalence classes of invertible $(n{-}1)$-supercategories$\} = \{$Morita equivalence classes of $(n{-}1)$-dimensional fermionic topological orders$\}$. Summarizing the results above, we have:
\begin{gather*} \pi_{-4} \underline{s\cW} = s\cW,  \quad
\quad \pi_{-3} \underline{s\cW} = 0, \quad
\pi_{-2} \underline{s\cW} = \bZ_2, \\
\pi_{-1} \underline{s\cW} = \bZ_2, \quad
\pi_0 \underline{s\cW} = \bC^\times.\end{gather*}
This spectrum $\underline{s\cW}$ is an extension to degree $-4$ of the spectrum $\SH$ of supercohomology, whose homotopy groups were $\pi_{-2} = \pi_{-1} = \bZ_2$ and $\pi_0 = \bC^\times$.

There is similarly a $4$-groupoid built from bosonic topological orders, corresponding to a spectrum $\underline{\cW}$. Its homotopy groups are
$$ \pi_{-3}\underline{\cW} = \cW, \qquad \pi_{-3} = \pi_{-2} = \pi_{-1} = 0, \quad \pi_0 = \bC^\times.$$
The relation between these two spectra is the following. $\underline{s\cW}$ carries a natural action by the ``$1$-form group'' $\rB\bZ_2$. Indeed, $\rB\bZ_2$ acts on the symmetric monoidal category $\SVec_\bC$ (c.f.\ \cite{MR3623677}), and so on the collection of all $n$-supercategories. (The only data in the action of $\rB\bZ_2$ on $\SVec_\bC$ is the action of the nontrivial $1$-morphism in $\rB\bZ_2$, which acts by the symmetric monoidal natural transformation of the identity functor ``$(-1)^f$'' that to a supervector space $\bC^{p|q}$ assigns the automorphism given by the $(p+q)\times(p+q)$ block matrix $\bigl( \begin{smallmatrix} +1 & \\ & -1 \end{smallmatrix}\bigr)$.) 

The homotopy fixed points of this $\rB\bZ_2$-action on the $4$-groupoid of fermionic topological orders is the $4$-groupoid of bosonic topological orders. Just as twisted-equivariant ordinary cohomology calculates the fixed points of a module, so too homotopy fixed points of actions on spectra are calculated by twisted-equivariant extraordinary cohomology.
Specifically, for $n \in \{-4,\dots,0\}$, 
$$ \pi_{-n}\underline{\cW} = \H^n_{\rB\bZ_2}(\pt; \underline{s\cW}).$$
The right-hand side may be calculated by using an Atiyah--Hirzebruch spectral sequence. (The spectral sequence converges because $\underline{s\cW}$ has bounded homotopy.)
$$ E^{p,q}_2 = \H^p_{\rB\bZ_2}(\pt; \pi_{-q} \underline{s\cW}) \Rightarrow \H^{p+q}_{\rB\bZ_2}(\pt; \underline{s\cW})$$
It is a standard fact that $\H^\bullet_{\rB\bZ_2}(\pt;\bZ_2)$ is freely generated over the Steenrod algebra by a generator ``$t$'' in degree $2$; as a polynomial algebra, it is $\bZ_2[t, \Sq^1t, \Sq^2\Sq^1t,\dots]$. The $\bC^\times$ cohomology can be computed from this, and so in low degree the $E_2$ page of the spectral sequence reads:
$$
\begin{tikzpicture}[anchor=base]
\path
(0,0) node {$\bC^\times$} ++(.75,0) node {$0$} ++(.75,0) node {$\bZ_2$} ++(.75,0) node {$0$} ++(.75,0) node {$\bZ_4$} ++(.75,0) node {$\bZ_2$}++(.75,0) node {$\cdots$}
(0,.5) node {$\bZ_2$} ++(.75,0) node {$0$} ++(.75,0) node {$\bZ_2$} ++(.75,0) node {$\bZ_2$} ++(.75,0) node {$\bZ_2$} ++(.75,0) node {$\bZ_2^2$}++(.75,0) node {$\cdots$}
(0,1) node {$\bZ_2$} ++(.75,0) node {$0$} ++(.75,0) node {$\bZ_2$} ++(.75,0) node {$\bZ_2$} ++(.75,0) node {$\bZ_2$} ++(.75,0) node {$\bZ_2^2$}++(.75,0) node {$\cdots$}
(0,1.5) node {$0$} ++(.75,0) node {$0$}++(.75,0) node {$0$}++(.75,0) node {$0$}++(.75,0) node {$0$}++(.75,0) node {$0$} ++(.75,0) node {$\cdots$}
(0,2) node {$s\cW$}  ++(.75,0) node {$0$}  ++(.75,0) node {$s\cW[2]$} ++(.75,0) node {$\cdots$}
;
\draw[->] (-.75,-.125) -- ++(5.25,0);
\draw[->] (-.4,-.5) -- ++(0,3);
\path  (0,-.5) node {$0$} ++(.75,0) node {$1$} ++(.75,0) node {$2$} ++(.75,0) node {$3$} ++(.75,0) node {$4$} ++(.75,0) node {$5$} ++(.75,0) node {$p$}
(-.75,0) node {$0$} ++(0,.5) node {$1$} ++(0,.5) node {$2$} ++(0,.5) node {$3$} ++(0,.5) node {$4$} ++(0,.5) node {$q$}
;
\end{tikzpicture}
$$
If there were no twisting, then the $d_2$ differentials would be
\begin{gather*}
d_2^{p,2} = \Sq^2 :  \H^p_{\rB\bZ_2}(\pt;\bZ_2) \to  \H^{p+2}_{\rB\bZ_2}(\pt;\bZ_2), \\
d_2^{p,1} = (-1)^{\Sq^2} : \H^p_{\rB\bZ_2}(\pt;\bZ_2) \to \H^{p+2}_{\rB\bZ_2}(\pt;\bC^\times).
\end{gather*}
Instead, the twisted differentials are
\begin{gather*}
 \tilde d_2^{p,2} = \Sq^2 + t, \\
 \tilde d_2^{p,1} = (-1)^{\Sq^2+ t}.
\end{gather*}
The nonzero differentials are:
$$
\begin{tikzpicture}[anchor=base]
\path
(0,0) node {$\bC^\times$} ++(.75,0) node {$0$} ++(.75,0) node (b2) {$\bZ_2$} ++(.75,0) node {$0$} ++(.75,0) node {$\bZ_4$} ++(.75,0) node {$\bZ_2$}++(.75,0) node (e2) {$\cdots$}
(0,.5) node (b1) {$\bZ_2$} ++(.75,0) node {$0$} ++(.75,0) node (a2) {$\bZ_2$} ++(.75,0) node {$\bZ_2$} ++(.75,0) node (e1) {$\bZ_2$} ++(.75,0) node (c2) {$\bZ_2^2$} ++(.75,0) node (d2) {$\cdots$}
(0,1) node (a1) {$\bZ_2$} ++(.75,0) node {$0$} ++(.75,0) node {$\bZ_2$} ++(.75,0) node (c1) {$\bZ_2$} ++(.75,0) node (d1) {$\bZ_2$} ++(.75,0) node {$\bZ_2^2$}++(.75,0) node {$\cdots$}
(0,1.5) node {$0$} ++(.75,0) node {$0$}++(.75,0) node {$0$}++(.75,0) node {$0$}++(.75,0) node {$0$}++(.75,0) node {$0$} ++(.75,0) node {$\cdots$}
(0,2) node {$s\cW$}  ++(.75,0) node {$0$}  ++(.75,0) node {$s\cW[2]$} ++(.75,0) node {$\cdots$}
;
\draw[->] (-.75,-.125) -- ++(5.25,0);
\draw[->] (-.4,-.5) -- ++(0,3);
\path  (0,-.5) node {$0$} ++(.75,0) node {$1$} ++(.75,0) node {$2$} ++(.75,0) node {$3$} ++(.75,0) node {$4$} ++(.75,0) node {$5$} ++(.75,0) node {$p$}
(-.75,0) node {$0$} ++(0,.5) node {$1$} ++(0,.5) node {$2$} ++(0,.5) node {$3$} ++(0,.5) node {$4$} ++(0,.5) node {$q$}
;
\draw[thick] 
(a1.mid) -- (a2.mid)
(b1.mid) -- (b2.mid)
(c1.mid) -- (c2.mid)
(d1.mid) -- (d2.mid)
(e1.mid) -- (e2.mid)
;
\end{tikzpicture}
$$
and so on the $E_3$ page we see:
$$
\begin{tikzpicture}[anchor=base]
\path
(0,0) node {$\bC^\times$} ++(.75,0) node {$0$} ++(.75,0) node {$0$} ++(.75,0) node {$0$} ++(.75,0) node {$\bZ_4$} ++(.75,0) node {$\bZ_2$}++(.75,0) node {$\cdots$}
(0,.5) node {$0$} ++(.75,0) node {$0$} ++(.75,0) node {$0$} ++(.75,0) node {$\bZ_2$} ++(.75,0) node {$0$} ++(.75,0) node  {$\cdots$}
(0,1) node {$0$} ++(.75,0) node {$0$} ++(.75,0) node {$\bZ_2$} ++(.75,0) node {$0$} ++(.75,0) node {$0$} ++(.75,0) node {$\cdots$}
(0,1.5) node {$0$} ++(.75,0) node {$0$}++(.75,0) node {$0$}++(.75,0) node {$0$}++(.75,0) node {$0$} ++(.75,0) node {$\cdots$}
(0,2) node {$s\cW$}  ++(.75,0) node {$0$} ++(.75,0) node {$\cdots$}
;
\draw[->] (-.75,-.125) -- ++(5.25,0);
\draw[->] (-.4,-.5) -- ++(0,3);
\path  (0,-.5) node {$0$} ++(.75,0) node {$1$} ++(.75,0) node {$2$} ++(.75,0) node {$3$} ++(.75,0) node {$4$} ++(.75,0) node {$5$} ++(.75,0) node {$p$}
(-.75,0) node {$0$} ++(0,.5) node {$1$} ++(0,.5) node {$2$} ++(0,.5) node {$3$} ++(0,.5) node {$4$} ++(0,.5) node {$q$}
;
\end{tikzpicture}
$$
Whatever is on the $E_\infty$ page in total degree $4$ will compile (via an extension) into the bosonic Witt group $\cW$. Thus, in order for the kernel of $\cW \to s\cW$ to be a $\bZ_{16}$, the $\bZ_2$s in bidegrees $(3,1)$ and $(2,2)$ must survive to $E_\infty$. The only question in this range of degrees is whether $\tilde d_5 : s\cW \to \bZ_2$ vanishes or not.
 If this $\tilde d_5$ does vanish, then $\cW \to s\cW$ is surjective, and the twisted equivariant cohomology $\H^5_{\rB\bZ_2}(\pt; \underline{s\cW})$ is a $\bZ_2$. If this $\tilde d_5$ is nonzero, then $\cW \to s\cW$ has cokernel of order $2$ (since its image is $\ker \tilde d_5$), and  $\H^5_{\rB\bZ_2}(\pt; \underline{s\cW})$ vanishes.
 
The following final remark about the differential $\tilde d_5$ is due to 
D.\ Nykshych. Each element in in $s\cW$ is represented by a braided fusion category $\cC$ which is ``slightly degenerate'' in the sense that its braided centre is $\SVec_\bC$. It is conjectured, but not known, that every slightly degenerate braided fusion category $\cC$ has a ``minimal modular extension'' $\cC \subset \cC'$, with $\cC'$ a (bosonic) braided fusion category with trivial centre, and with $\cC'$ having twice the total dimension of $\cC$. 
It is not too hard to show that if $\cC$ has a minimal modular extension, then its Morita-equivalents do too, so that existence of $\cC'$ is a question about the class of $\cC$ in $s\cW$. Moreover, it is not too hard to show that the only possible obstruction to the existences of a minimal modular extension lives in $\H^5_{\rB\bZ_2}(pt; \bC^\times) \cong \bZ_2$. Finally, it is not too hard to show that $\cW \to s\cW$ is surjective if and only if every slightly degenerate braided fusion category admits a minimal modular extension. All together, we find that the differential  $\tilde d_5 : s\cW \to \H^5_{\rB\bZ_2}(pt; \bC^\times)= \bZ_2$ records the obstruction to minimal modular extensions: it sends a class $[\cC] \in s\cW$ to $0 \in \bZ_2$ if a minimal modular extension $\cC' \supset \cC$ exists, and to $1 \in \bZ_2$ if a minimal modular extension $\cC'$ does not exist. Thus the ``minimal modular extension conjecture'' (asserting that $\cC'$ always exists) is equivalent to the conjecture that $\tilde d_5 = 0$, or that $\H^5_{\rB\bZ_2}(\pt; \underline{s\cW}) = \bZ_2$.
\end{remark}

\subsection{$(3{+}1)$-dimensional topological orders} \label{subsec.3+1d}

A classification of $(3{+}1)$-dimensional topological orders was announced in \cite{PhysRevX.8.021074,PhysRevX.9.021005,PhysRevB.100.045105}. We will confirm that classification in this section, modulo a few small corrections and improvements.

Our strategy will be to repeat the analysis from \S\ref{subsec.1+1d} of $(1{+}1)$-dimensional topological orders, where we used the fact that $\Omega\cA$, being a separable commutative $\bC$-algebra, was canonically determined (as the algebra of functions) by its \define{spectrum} $\Spec(\Omega\cA) = \hom(\Omega\cA,\bC)$, the set of homomorphisms of commutative $\bC$-algebras. In the presence of a $G$-flavour symmetry, this spectrum was a $G$-set, and we used \emph{$G$-equivariant} reduced cohomology to classify $\cA$. In the presence of a $\bZ_2^T$ time-reversal symmetry, we could continue to think of $\Spec(\Omega\cA)$ as a $\bZ_2^T$-set, but we could also think of it as a \define{real scheme} $\Spec(\Omega\cA_\bR)$, where $\Omega\cA_\bR \subset \Omega\cA$ is the $\bR$-subalgebra fixed by the $\bZ_2^T$ symmetry; thinking in the latter way, the twisted equivariant cohomology became \define{Galois cohomology}.

It $\cA$ is instead a $(3{+}1)$-dimensional topologica lorder, we will work with $\Omega^2\cA$ in place of $\Omega\cA$. This 
 is a symmetric multifusion category, which is a natural categorification of ``separable commutative algebra.'' There is a theory of spectra of such categories, which is essentially a special case of P.\ Deligne's work on Tannakian reconstruction \cite{zbMATH03749168,MR1106898,MR1944506}; the interpretation in terms of ``categorical Galois extensions'' is due to \cite{MR3623677}. Summarizing a number of results, the statement of Tannakian reconstruction in the multifusion category case is:

\begin{theorem}\label{thm.Deligne}
A symmetric multifusion supercategory $\cC$ is canonically determined (as the category of representations) by its \define{spectrum} $\Spec(\cC) = \hom(\cC, \SVec_\bC)$, the finite groupoid of functors of symmetric monoidal supercategories. The set $\pi_0\Spec(\cC)$ of isomorphism classes of objects in this groupoid is $\Spec(\Omega\cC)$.

If $\cC$ is bosonic, then $\Spec(\cC)$ is equipped with an action of the \define{categorified Galois group} $\Gal(\SVec_\bC/\Vect_\bC) = \rB\bZ_2$, and $\cC$ is recovered canonically from $\Spec(\cC)$ with this action. \qed
\end{theorem}

The group $\rB\bZ_2$ is a group object in groupoids, not in sets, i.e.\ it is a \define{2-group}. It has only one object, and so a physicist might also call it a ``group of 1-form symmetries.'' The ``Galois'' action of $\rB\bZ_2$ on $\Spec(\cC)$ should be understood as \define{descent data}, making $\Spec(\cC)$ into a \define{categorified scheme over $\Vect_\bC$}. The fact that $\Spec(\cC)$ is merely a groupoid in the super case, but requires descent data in the bosonic case, is interpreted in \cite{MR3623677} to mean that $\SVec_\bC$, and not $\Vect_\bC$, is ``algebraically closed.''

\begin{remark}
  In positive characteristic, not all multifusion categories are separable (see Remark~\ref{remark.multifusion}). With the word ``separable'' added everywhere to the phrase ``symmetric multifusion category,'' Theorem~\ref{thm.Deligne} remains true in positive characteristic, and is due to~\cite{Ostrik2015}, and so a version of Theorem~\ref{thm.3+1d} below applies in all characteristics.
  Theorem~\ref{thm.Deligne} fails for inseparable fusion categories: counterexamples, and the corrected statement, are given in \cite{Ostrik2015}.
\end{remark}

Returning to our goal of classifying topological orders,
suppose that $\cA$ is a $(3{+}1)$-dimensional fermionic topological order. (We will handle the fermionic case first, because Theorem~\ref{thm.Deligne} is simpler in that case.) Then $\Omega^2\cA$ is a symmetric multifusion supercategory, and so the category of representations $\Spec(\Omega^2\cA) \to \SVec_\bC$ for some canonically-defined finite groupoid $\Spec(\Omega^2\cA)$. Choose any point $s : \pt \to \Spec(\Omega^2\cA)$. In the Tannakian language, $s$ is a choice of \define{fibre functor} $\Omega^2\cA \to \SVec_\bC$.
(Let us emphasize that a choice of $s$ is more than just a choice of its isomorphism class. Suppose for instance that $\cA$ is fusion. Then all fibre functors $s$ are isomorphic, but not canonically so: the ambiguity in choosing the isomorphism is precisely the group that Tannakian duality reconstructs in the fusion case.)

Using $s$, one may construct in $\cA$ a condensable algebra, just as in the $(1{+}1)$-dimensional case. Indeed, one step in the proof of Theorem~\ref{thm.Deligne} is to convert $s$ into an algebra object $\Delta_s \in \Omega^2\cA$ (compare also Corollary~7.10.5 of~\cite{EGNO} or Theorem~1 of~\cite{MR1976459} identifying module categories with categories of modules). This algebra object is in fact commutative, and so $\Sigma\Delta_s \in \Omega\cA$ is a braided algebra object, and $\Sigma^2\Delta_s \in \cA$ is an associative algebra object. The condensation of this algebra is a new $(3{+}1)$-dimensional topological order separated from $\cA$ by an interface, and so Morita-equivalent to $\cA$ (as shown in the proof of Theorem~\ref{thm.anomalousTQFT}). The new topological order has as its category of lines the category of $\Delta_s$-modules, which is precisely $\SVec_\bC$. Thus the new topological order is trivial.

It follows that $\cA = \End(\cM_s)$ for a $3$-supercategory $\cM_s$ produced from the condensation procedure. As in the $(1{+}1)$-dimensional case, this $3$-category can be described explicitly. It is ``graded'' by $\Spec(\Omega^2\cA)$ in the following sense. Given any $s' \in \Spec(\Omega^2\cA)$, one can write down a $3$-category $\cM_{s,s'}$ of interfaces between the $s$ and $s'$ local ground states. This $3$-category depends functorially in $s'$, i.e.\ it is a bundle of $3$-categories over $\Spec(\Omega^2\cA)$. The category $\cM_s$ is the ``direct sum'' of this bundle, defined as either a limit or colimit (the two being canonically equivalent; compare \cite{MR3221291}):
$$ \cM_s = \bigoplus_{s' \in \Spec(\Omega^2\cA)} \cM_{s,s'}.$$
Then comparing the decompositions
\begin{align*}
 \End(\cM_s) & = \bigoplus_{s',s'' \in \Spec(\Omega^2\cA)} \hom(\cM_{s,s'},\cM_{s,s''}), \\
 \cA & = \bigoplus_{s',s'' \in \Spec(\Omega^2\cA)} \cM_{s',s''},
\end{align*}
we learn that each $\cM_{s,s'}$ is an invertible $3$-supercategory.

Invertible $3$-supercategories, up to equivalence, are precisely the Morita equivalence classes (i.e.\ the gravitational anomalies) of $(2{+}1)$-dimensional fermionic topological orders. In \S\ref{subsec.2+1d} we saw that every such Morita equivalence class is represented by a braided fusion supercategory: as a set they form the fermionic Witt group $s\cW$, and as a higher groupoid they form the spectrum $\underline{s\cW}$ studied in Remark~\ref{remark.fermioniocWittgroup}. The upshot is that the assignment $s' \mapsto \cM_{s,s'}$ determines a class in $\H^4(\Spec(\Omega^2\cA); \underline{s\cW})$. This class depended on the choice of basepoint $s$, and so:

\begin{theorem}\label{thm.3+1d}
  A fermionic $(3{+}1)$-dimensional topological order $\cA$ is classified by a finite groupoid $\Spec(\Omega^2\cA)$ together with a reduced extraordinary cohomology class in $\widetilde{\H}{}^4(\Spec(\Omega^2\cA); \underline{s\cW})$. \qed
\end{theorem}

Theorem~\ref{thm.3+1d} can be summarized by saying that $(3{+}1)$-dimensional topological orders are, canonically, \define{anomalous sigma models} with target finite groupoids. The sigma model built from a finite groupoid $\cX$ and a reduced extraordinary cohomology class in $\widetilde{\H}{}^4(\cX; \underline{s\cW})$ is ``anomalous'' in the sense that it depends only on a reduced cohomology class. In particular, if the target space $\cX$ is equipped with a $G$-symmetry, then $G$-equivariant anomalous sigma models are classified by reduced extraordinary $G$-equivariant cohomology $\widetilde{\H}{}^4_G(\cX; \underline{s\cW})$, and any class therein determines a (typically nonzero) \define{'t Hooft anomaly} living in the equivariant extraordinary cohomology $\H^5_G(\pt;\underline{s\cW})$.

An important special case is:

\begin{corollary}\label{cor.3+1nondegenerate}
  Fermionic $(3{+}1)$-dimensional topological order with nondenegerate local ground states are, canonically, gauge theories for finite groups. The gauge group is $\pi_1 \Spec(\Omega^2\cA)$, and the fermionic Dijkgraaf--Witten action is a class in $\SH^4(\rB G)$.
\end{corollary}
\begin{proof}
  The topological order has a nondegenerate local ground state exactly when $\Spec(\Omega^2\cA)$ is connected, in which case it is the classifying groupoid of its fundamental group. In the connected case, $\widetilde{\H}{}^4(\Spec(\Omega^2\cA); \underline{s\cW}) = \SH^4(\Spec(\Omega^2\cA))$ because of the quotient by $\H^4(\pt; \underline{s\cW}) = s\cW$.
\end{proof}

Theorem~\ref{thm.3+1d} is true also in the presence of symmetry, just as in the $(1{+}1)$-dimensional case studied in Theorem~\ref{thm.1+1d}: $(3{+}1)$-dimensional topological orders with $G$-symmetry are classified by reduced $G$-equivariant cohomology with coefficients in $\underline{s\cW}$. Of particular importance is when the symmetry group is $\rB\bZ_2$, acting nontrivially on the coefficients $\underline{s\cW}$ as $(-1)^f$. A categorified version of Galois descent says that fermionic topological orders with this $\rB\bZ_2 = (-1)^f$ action are precisely bosonic topological orders.
(This ``categorified Galois descent'' perspective is due to \cite{MR3623677}.)
 Thus we find:

\begin{corollary}
  A bosonic $(3{+}1)$-dimensional topological order $\cA$ is classified by a finite groupoid $\Spec(\Omega^2\cA)$ together with an action of $\rB\bZ_2$ on $\Spec(\Omega^2\cA)$ and a class in reduced twisted equivariant extraordinary cohomology $\widetilde{\H}{}^4_{\rB\bZ_2}(\Spec(\Omega^2\cA); \underline{s\cW})$. \qed
\end{corollary}

  If there is a nondegenerate local ground state, so that $\Spec(\Omega^2\cA) = \rB G$ for a finite group $G$, then the classification restricts to a class in twisted equivariant supercohomology $\SH^4_{\rB\bZ_2}(\rB G)$, just as in Corollary~\ref{cor.3+1nondegenerate}. 
It is worth unpacking this a bit. An action of the $1$-form group $\rB\bZ_2$ on $\rB G$ is equivalent to the choice of a central order-$2$ element $\epsilon \in G$. The pair $(G,\epsilon)$ is sometimes called a ``finite supergroup,'' and shows up in the Tannakian description of bosonic symmetric fusion categories. In particular, as a bosonic symmetric fusion category, $\Omega^2\cA$ is precisely the category of super $G$-modules in which $\epsilon$ acts by $(-1)^f$.

When $\epsilon$ is the identity in $G$, so that the action of $\rB \bZ_2$ on $\rB G$ is trivial, then the twisted equivariant supercohomology $\SH^4_{\rB\bZ_2}(\rB G)$ reduces to the ordinary cohomology $\H^4(\rB G; \bC^\times)$, and we recover the classification of \cite{PhysRevX.8.021074}. Otherwise, the bosonic topological order has ``emergent fermions,'' and we recover the classification from \cite{PhysRevX.9.021005}.

\begin{remark}\label{remark.higherdimensions}
  One can try to extend this style of analysis to higher dimensions. For example, if $\cA$ is a $(5{+}1)$-dimensional topological order, then $\Omega^3\cA$ is a symmetric multifusion $2$-category. In unpublished work, M.\ Hopkins and the author have established the $2$-categorification of Theorem~\ref{thm.Deligne}: a  symmetric multifusion $2$-supercategory $\cC$ is determined by its spectrum $\Spec(\cC) = \hom(\cC, \Sigma\SVec_\bC)$, which is a finite $2$-groupoid. This implies that $(5{+}1)$-dimensional topological orders are ``anomalous sigma models'' with target a finite $2$-groupoid, but with a complicated cohomology theory classifying the sigma model action: whereas in Theorem~\ref{thm.3+1d} we needed the spectrum $\underline{s\cW}$ of Morita equivalence classes of topological orders of dimension $\leq (2{+}1)$, for the $(5{+}1)$-dimensional case we would need to understand invertible fermionic topological orders of dimension $\leq (4{+}1)$. 
  
  This is not hopeless. Using the methods above, one can reduce to the case of classifying, up to Morita equivalence, $(4{+}1)$-dimensional topological orders with neither point nor line operators, which is to say Morita-invertible $3$-monoidal $2$-categories. It is conceivable that a Morita-invertible $3$-monoidal $2$-category is necessarily ``abelian'' in the sense that its simple objects form an abelian group $X$. If $X$ has odd order, then the nondegenerate $3$-monoidal structure is precisely a symplectic form on $X$, and a choice of Lagrangian subgroup $L \subset X$ provides a Morita-trivialization of the corresponding topological order. Thus one can hope that the Morita-equivalence classes of $(4{+}1)$-dimensional topological orders will  form some group of order a small power of $2$. See \cite{2104.04534} for further details.
    
  In $(7{+}1)$-dimensions, one would need a $3$-categorification of Theorem~\ref{thm.Deligne}. There are counterexamples to the most obvious $3$-categorification, which would try to recover a symmetric multifusion $3$-supercategory from its $3$-groupoid of maps to $\Sigma^2\SVec_\bC$. These counterexamples are analogues of the fact that a symmetric multifusion $1$-category is not recoverable from its $1$-groupoid of maps to $\Vect_\bC = \Sigma\bC$, 
  but rather one needs to extend to $\SVec_\bC$.
   The correct $3$-categorification of Theorem~\ref{thm.Deligne} remains unclear.
\end{remark}


\newcommand{\etalchar}[1]{$^{#1}$}

\end{document}